\documentclass[12pt,reqno]{amsart}
\usepackage[english]{babel}
\usepackage{amsmath}
\usepackage{amsthm}
\usepackage{anysize}
\usepackage{mathtools}
\usepackage{mathrsfs}
\usepackage{amssymb}
\usepackage{xcolor}
\usepackage{xfrac}
\usepackage{esint}
\usepackage{eucal}
\usepackage{cite}
\usepackage{graphicx}
\usepackage{float}
\usepackage{array}
\usepackage{enumitem}
\usepackage{tikz-cd}
\usepackage{centernot}
\usepackage{stmaryrd}
\usepackage{comment}
\excludecomment{confidential}
\usepackage{epstopdf}
\usepackage{mathabx}

\usepackage{graphicx,color}
\usepackage{tikz,pgfplots,pgf}
\usepackage[font=small]{caption}

\newtheorem{theorem}{Theorem}[section]
\newtheorem{lemma}{Lemma}[section]

\newtheorem{corollary}{Corollary}[section]
\theoremstyle{definition}

\theoremstyle{definition}

\newtheorem{definition}{Definition}[section]

\usepackage[english]{babel}
\usepackage{amsmath}
\usepackage{amsthm}
\usepackage{mathtools}
\usepackage{mathrsfs}
\usepackage{amssymb}
\usepackage{xcolor}
\usepackage{xfrac}
\usepackage{esint}
\usepackage{graphicx}
\usepackage{float}
\usepackage{array}
\usepackage{enumitem}
\usepackage[new]{old-arrows}
\usepackage[all,cmtip]{xy}
\usepackage{tikz-cd}
\usepackage{centernot}
\usepackage{stmaryrd}
\usepackage{comment}
\usepackage{bbm}

\newcommand{\mc}{\mathcal}
\newcommand{\mf}{\mathfrak}

\newcommand{\K}{\mathbb{K}}
\newcommand{\R}{\mathbb{R}}
\newcommand{\N}{\mathbb{N}}
\newcommand{\C}{\mathbb{C}}

\renewcommand{\l}{\lambda}
\renewcommand{\O}{\Omega}
\renewcommand{\d}{\delta}
\renewcommand{\a}{\alpha}
\renewcommand{\b}{\beta}
 \renewcommand{\o}{\omega}

\numberwithin{equation}{section}

\makeatletter
\@namedef{subjclassname@2020}{\textup{2020} Mathematics Subject Classification}
\makeatother

\def\a{\alpha}
\def\b{\beta}
\def\e{\varepsilon}
\def\D{\Delta}
\def\d{\delta}
\def\g{\gamma}
\def\G{\Gamma}
\def\l{\lambda}

\def\o{\omega}
\def\O{\Omega}
\def\p{\partial}

\def\s{\sigma}

\def\ua{\uparrow}
\def\da{\downarrow}

\theoremstyle{definition}

\theoremstyle{remark}

\newcommand{\Z}{\mathbb{Z}}

\newcommand{\ms}{\mathscr}

\begin{document}

\title{Global Leray--Schauder continuation for Fredholm operators}

\author[J. López-Gómez]{Julián López-Gómez}\thanks{This work has been supported by the Ministry of Science and Innovation of Spain under the
Research Grant PID2024–155890NB-I00 and by the Institute of Interdisciplinary Mathematics of the Complutense University of Madrid}
\address{Departamento de Análisis Matemático y Matemática Aplicada\\
	Instituto de Matem\'{a}tica Interdisciplinar \\
	Plaza de las Ciencias 3 \\
	Universidad Complutense de Madrid (UCM) \\
	Madrid, 28040, Spain.}
\email{jlopezgo@ucm.es}

\author[J. C. Sampedro]{Juan Carlos Sampedro} \thanks{}
\address{Departamento de Matemática Aplicada y Ciencias de la Computación \\
	Instituto de Matem\'{a}tica Interdisciplinar \\
	Avenida de los Castros 46 \\
	Universidad de Cantabria (UC) \\
	Santander, 39005, Spain.}
\email{juancarlos.sampedro@unican.es}

\begin{abstract}
This paper ascertains the global behavior of the forward and backward \lq\lq branches\rq\rq\, of solutions provided by the Leray--Schauder continuation theorem for orientable $\mathcal{C}^1$ Fredholm maps, as developed by the authors in \cite{LGS24}. Under properness on bounded sets and a nonzero local index at the given base solution, each branch satisfies the following alternative: either it is unbounded, or it reaches the boundary of the domain, or it accumulates at a different solution on the base parameter level. When the component is bounded and stays
in the interior, there is a degree balance on the base slice entailing a vanishing sum of local indices and, in particular, the existence of an even number of non-degenerate contact points. For real-analytic maps we construct locally injective parameterizations that exhibit blow-up, approach to the boundary, or return to the base level. An application to a quasilinear boundary value problem driven by the mean–curvature and Minkowski operators illustrates the global results.
\end{abstract}

\keywords{Analytic operators, Continuation theorems, Fredholm operators, Mean curvature equation, Topological degree}
\subjclass[2020]{46T20, 47H11 (primary), 47A53, 46G12, 34K18 (secondary)}

\maketitle

\tableofcontents

\section{Introduction}

\noindent
This paper discusses the classical continuation principle of Leray and Schauder, who, in 1934, in their seminal paper \cite{LeS}, introduced a degree theory for compact perturbations of the identity map in a Banach space $U$ and used it to develop a general method for solving nonlinear functional and differential equations. Roughly speaking, the Leray--Schauder continuation theorem asserts that if one can embed a given fixed point equation for a compact operator, $K$,  into a
set of fixed point equations for a one–parameter family of compact operators, $K(\l)$, in such a way that its solution set satisfies  some non-degeneration condition at some value of the parameter, $\l=\l_0$, and the solutions of $u=K(\l)u$ cannot escape through the boundary of $[a,b]\times\O$, where $a<b$, $\l_0\in[a,b]$, and
$\O$ is a bounded open subset of $U$, then, $u=K(\l)u$ admits a continuum of solutions intersecting $\{\l\}\times \O$ for all $\l\in [a,b]$. Since then,  this result has become one of the cornerstones of nonlinear functional analysis and global bifurcation theory (see, e.g., the surveys and monographs \cite{Ma1,KZ,Ze}). By a continuum we mean any closed and connected subset.
\par
However, in practice,  many nonlinear boundary value problems do not fit naturally into this compact perturbation framework, but, rather, into the setting of the more flexible degree theory for Fredholm operators of index zero. Within such context, the degree of Fitzpatrick, Pejsachowicz and Rabier \cite{FP91,FPR,PR} provides an extension of the Leray--Schauder degree, which is defined for orientable $\mathcal{C}^1$ Fredholm maps of index zero, it is stable under proper Fredholm homotopies, and it preserves the key properties supporting the Leray--Schauder continuation theorem.
\par
Essentially, in this paper we are investigating the global behavior of the unilateral components of the continuum given by the Leray--Schauder continuation theorem delivered by the authors in Theorem~3.8 of \cite{LGS24} for orientable $\mathcal{C}^1$ Fredholm maps. Let $U$ and $V$ be real Banach spaces, let $\mathcal{U}\subset \mathbb{R}\times U$ be open, and let $\mathfrak{F}:\mathcal{U}\to V$ be an orientable Fredholm map of class $\mathcal{C}^1$, which is proper on every bounded open subset of $\mathcal{U}$. Suppose $(\lambda_0,u_0)\in\mathcal{U}$ satisfies $\mathfrak{F}(\lambda_0,u_0)=0$, and the Fredholm local index of the slice map $\mathfrak{F}_{\lambda_0}$, $u\mapsto \mathfrak{F}(\lambda_0,u)$, at $u_0$ is nonzero. For $\lambda\in\mathbb{R}$ and $\Omega\subset\mathbb{R}\times U$ we write
\[
\Omega_\lambda:=\{u\in U:\; (\lambda,u)\in \Omega\}.
\]
Then, there exists a connected component $\mathscr{C}$ of $\mathfrak{F}^{-1}(0)$ passing through $(\lambda_0,u_0)$ such that its unilateral parts are non-empty, i.e.
\begin{equation}\label{eq:unilateral}
	\mathscr{C}^+\!:=\{(\lambda,u)\in\mathscr{C}:\lambda>\lambda_0\}\neq\emptyset,
	\qquad
	\mathscr{C}^-\!:=\{(\lambda,u)\in\mathscr{C}:\lambda<\lambda_0\}\neq\emptyset.
\end{equation}
When $D_u\mathfrak{F}(\lambda_0,u_0)\in GL(U,V)$, the implicit function theorem ensures that $\mathscr{C}$ is a $\mathcal{C}^1$ curve in a neighborhood of $(\lambda_0,u_0)$. This paper focuses attention on the \emph{global} alternatives satisfied by the unilateral components $\mathscr{C}^{\pm}$.
\par
More precisely, the first goal of this paper is to use the topological degree for Fredholm operators discussed
by the authors in \cite{LGS24} to ascertain the global behavior of each of the unilateral components $\ms{C}^\pm$. It turns out that each of these components, $\ms{C}^\pm$, satisfies some of the following alternatives:
\begin{enumerate}
		\item[A1.] $\mathscr{C}^{\pm}$ is unbounded in $\mc{U}$.
		\item[A2.] $\mathscr{C}^{\pm}\cap \partial\mc{U}\neq \emptyset$.
		\item[A3.] There exists $(\l_0,u_{1})\in\mc{U}_{\l_{0}}$, with $u_{1}\neq u_{0}$,  such that $(\l_{0},u_{1})\in \overline{\mathscr{C}}^{\pm}$.
	\end{enumerate}
Moreover, if $\mathscr{C}^{\pm}$ is bounded, $\mathscr{C}^{\pm}\cap \partial\mc{U}=\emptyset$, and
$$
   \mf{F}^{-1}_{\l_{0}}(0)\cap \big(\overline{\ms{C}}^\pm\big)_{\l_{0}} =\{u_0,u_1,\ldots,u_q\}
$$
for some integer $q\geq 1$, then
\begin{equation}
\label{iii.3}
	\sum_{j=0}^{q} i(\mf{F}_{\l_{0}},u_{j},\varepsilon_{\l_{0}})=0,
\end{equation}
where for any given
orientation, $\e$,  of $D_u\mf{F}$ in $\mc{O}$, we are denoting by $\e_{\l_0}$ the frozen orientation
$\e_{\l_0}(u)=\e(\l_0,u)$. In particular, there are $2\nu$, $\nu\geq 1$, points $u\in \mf{F}^{-1}_{\l_{0}}(0)\cap (\overline{\ms{C}}^\pm)_{\l_{0}}$ such that $i(\mf{F}_{\l_{0}},u,\varepsilon_{\l_{0}})\neq 0$. Therefore,
when  $\mathscr{C}^{\pm}$ is bounded and $\mathscr{C}^{\pm}\cap \partial\mc{U}= \emptyset$, our continuation theorem gives more detailed information than Theorem 1 of Santos et al. \cite{SCR}. These results are based on some ideas going back to Theorem 3.5 of Rabinowitz \cite{Ra71b} (see also Arcoya et al. \cite{ACJT})  in the context of the Leray--Schauder degree, where technicalities are far less sophisticated than in the framework of the Fitzpatrick--Pejsachowicz--Rabier degree for Fredholm operators (see \cite{FP91}--\cite{FPR} and \cite{PR}) dealt with in this paper.
\par
The second goal of this paper is to show that if, in addition, $\mf{F}(\l,u)$ is real analytic in $\l$ and $u$, then, there exist $\o^\pm\in \N\cup\{+\infty\}$ and two locally injective continuous curves,
$$
     \G^{\pm}: (0,\o^\pm)\longrightarrow \mc{U}_{\l_{0}}^{\pm}:=\{(\l,u)\in\mc{U}\;:\;\pm(\l-\l_0)>0\},
$$
such that
$$
   \Gamma^{\pm}\left((0,\o^\pm)\right)\subset \mf{F}^{-1}(0), \qquad \lim_{t\da 0}\Gamma^{\pm}(t)=(\l_{0},u_{0}),
$$
for which some of the following non-excluding options occurs:
\begin{enumerate}
		\item[{\rm (a)}]   The curve $\G^\pm$ blows up at $\omega^\pm$, in the sense that
$$
    \limsup_{t\ua \o^\pm} \|\Gamma^{\pm}(t)\|_{\mathbb{R}\times U}= +\infty.
$$
\item[{\rm (b)}]   The curve $\G^\pm$ approximates $\p \mc{U}$ as $t\to \o^\pm$, in the sense that
there exists a sequence $\{t_{n}\}_{n\in\N}$ in $(0,\o^\pm)$ such that
$$
   \lim_{n\to +\infty}t_{n} = \o^\pm \;\;\hbox{and}\;\;  \lim_{n\to+\infty} \Gamma^{\pm}(t_{n})= (\l_{\ast},u_{\ast})\in \partial\mc{U}.
$$
		\item[{\rm (c)}]   The curve $\G^\pm$ turns backwards to the level $\l=\l_0$, in the sense that there exist $u_{1}\in \mc{U}_{\l_{0}}\setminus\{u_{0}\}$ and a sequence $\{t_{n}\}_{n\in\N}$ in $(0,\o^\pm)$ such that
$$
   \lim_{n\to +\infty} t_{n} = \o^\pm\;\; \hbox{and}\;\; \lim_{n\to+\infty}\Gamma^{\pm}(t_{n})=(\l_{0},u_{1})\in\mf{F}^{-1}(0).
$$
\end{enumerate}
This result is based on the theorem of structure of real analytic varieties, which has been revisited in Theorem \ref{thA.1} of the Appendix, and on some technical tools developed by Buffoni and Toland in \cite{BT} and its references.
\par
The organization of this paper is the following. Section 2 collects some elements on the degree for Fredholm operators and delivers the generalized homotopy invariance property used through this paper. Section 3 discusses the global behavior of the components $\ms{C}^\pm$ through the technical devices developed by the authors in Theorem 6.3.1 of \cite{LG01}, in the context of the Leray--Schauder degree,  and Theorem 5.9 of \cite{LGS24},  in the context of the degree for Fredholm operators. The main technical device to prove
our main continuation theorem in Section 3 is the construction of open isolating neighborhoods for the
unilateral components $\ms{C}^\pm$. This construction is based on a celebrated theorem of Whyburn \cite{Wh} on
continua. In Section 4, we show, with a particular example, how the existence of open isolating neighborhoods relies on the compactness properties of the underlying fixed point operators. In Section 5 we deliver our main analytic continuation theorems discussed before. Section~6 applies the abstract results developed in the preceding sections to a quasilinear parameter-dependent boundary value problem driven by the mean–curvature operator and the Minkowski operator. As these operators do not admit a convenient inversion yielding to a compact perturbation of the identity, the Fredholm framework is particularly well suited to this setting. Finally, Appendix~A summarizes some basic facts on real-analytic varieties borrowed from Buffoni and Toland~\cite{BT}.
\par
Throughout this paper, for any given linear operator $T$ between two Banach spaces, we denote by $N[T]$ and $R[T]$ the null space (or kernel) and the range (or image) of $T$, respectively. Moreover,
the notation $\uplus$ stands for the disjoint union.

\section{Preliminaries}

\noindent This section collects some basic concepts and results that are going to be used throughout this paper. In particular, we recall the definition of the topological degree for Fredholm operators of index zero. We begin by reviewing the notions of \emph{parity} and \emph{orientability}.

\subsection{Parity and Orientability}

In this section, $(U,V)$ is a pair of real Banach spaces and $\Phi_n(U,V)$ stands for the set of Fredholm operators $T: U\to V$ of index $n\in\N$. A continuous curve of operators
$\mathscr{L}\in\mathcal{C}([a,b],\Phi_{0}(U,V))$
is said to be \emph{admissible} if its endpoints are invertible, i.e.
$\mathscr{L}(a)$, $\mathscr{L}(b)\in GL(U,V)$.
For any given admissible curve $\mathscr{L}(\l)$, a \emph{parametrix} is a curve
$\mathfrak{P}\in\mathcal{C}([a,b],GL(V,U))$  such that
\[
\mathfrak{P}(\lambda)\,\mathscr{L}(\lambda)-I_{U}\in\mathcal{K}(U)\quad \hbox{for all}\;\; \lambda\in[a,b],
\]
where $\mathcal{K}(U)$ denotes the space of compact operators on $U$ and $I_{U}:U\to U$ the identity operator.
The existence of a parametrix is guaranteed by Theorem~2.1 of Fitzpatrick and Pejsachowicz~\cite{FP91}. For any given admissible curve $\mathscr{L}(\l)$, its \emph{parity} in $[a,b]$ is defined by
\begin{equation*}
	\sigma(\mathscr{L},[a,b])
	:=\deg_{LS}(\mathfrak{P}(a)\mathscr{L}(a))\cdot
	\deg_{LS}(\mathfrak{P}(b)\mathscr{L}(b)),
\end{equation*}
where $\mathfrak{P}$ is any parametrix of $\mathscr{L}$, and
$$
    \deg_{LS}(T):=\deg_{LS}(T,B_{\varepsilon})
$$
stands for the Leray--Schauder degree of the operator $T$
in the ball of radius $\e>0$ centered at $0$, $B_\e$, for sufficiently small $\varepsilon>0$. This definition is consistent, since $\sigma$ is independent of the chosen parametrix.
\par
The main obstacle in defining a topological degree for Fredholm operators of index zero is the absence of a canonical orientation in $GL(U,V)\subset\Phi_{0}(U,V)$ (see e.g. Kuiper \cite{Ku}).
The classical approach of Fitzpatrick, Pejsachowicz and Rabier~\cite{FPR}
restricts attention to those maps for which a notion of \emph{orientability} can be introduced.
\par
Let $X$ be a path-connected topological space and let $h:X\to \Phi_{0}(U,V)$ be a continuous map.
A point $x\in X$ is said to be \emph{regular with respect to $h$} if $h(x)\in GL(U,V)$. In this paper,
we denote by $\mathscr{R}_{h}$ the set of regular points of $h$, i.e.
$$
   \mathscr{R}_h:=h^{-1}(GL(U,V)).
$$

\begin{definition}
	\label{deb.1}
Let $X$ be a path-connected topological space and $(U,V)$ a pair of real Banach spaces.
A continuous map $h:X\to\Phi_{0}(U,V)$ is said to be \emph{orientable} if there exists a function
\(\varepsilon:\mathscr{R}_{h}\to\mathbb{Z}_{2}\), called an orientation, such that, for every continuous curve 	 $\gamma\in\mathcal{C}([a,b],X)$ with $\gamma(a),\gamma(b)\in \mathscr{R}_{h}$,  one has that
\begin{equation}
		\label{ii.1}
		\sigma(h\circ\gamma,[a,b])=\varepsilon(\gamma(a))\cdot \varepsilon(\gamma(b)).
\end{equation}
Should $X$ have several path-connected components, $h$ is said to be orientable if it is orientable on each of them.
\end{definition}

\subsection{Degree for Fredholm operators}
First, we introduce the class of operators for which the degree is defined.
Let $\mathcal{O}\subset U$ be open, $n\in\mathbb{Z}$, and $r\in\mathbb{N}$.
An operator $f:\mathcal{O}\to V$ is called $\mathcal{C}^{r}$-\emph{Fredholm of index $n$} if
$$
   f \in\mathcal{C}^{r}(\mathcal{O},V)\;\;\hbox{and}\;\; Df\in \mathcal{C}^{r-1}(\mathcal{O},\Phi_{n}(U,V)).
$$
Throughout this paper, the collection of these maps is denoted by $\mathscr{F}^{r}_{n}(\mathcal{O},V)$. For any given $f\in \mathscr{F}^{r}_{0}(\mathcal{O},V)$, it is said that $f$ is \emph{orientable} if
$Df:\mathcal{O}\to \Phi_{0}(U,V)$ is orientable.
\par
Let $\Omega\subset U$ be an open and bounded subset of $U$ such that $\overline{\Omega}\subset\mathcal{O}$, and suppose that the  operator $f:\mathcal{O}\to V$ satisfies the following conditions:
\begin{enumerate}
	\item $f\in \mathscr{F}^{1}_{0}(\mathcal{O},V)$ is orientable, with orientation
	$\varepsilon:\mathscr{R}_{Df}\to\mathbb{Z}_{2}$;
	\item $f$ is proper in $\overline{\Omega}$, i.e. $f^{-1}(K)\cap \overline{\O}$ is compact for every compact subset $K\subset V$;
	\item $0\notin f(\partial\Omega)$.
\end{enumerate}
Then, it is said that $(f,\Omega,\varepsilon)$ is a \emph{Fredholm $\mathcal{O}$-admissible triple}.
The class of all these triples is denoted by $\mathscr{A}(\mathcal{O})$.
\par
Next, we introduce the set of admissible homotopies. A map $H\in \mathcal{C}^{r}([a,b]\times \mathcal{O},V)$ is said to be a $\mathcal{C}^{r}$-\emph{Fredholm homotopy} if $H\in \mathscr{F}^{r}_{1}([a,b]\times \mathcal{O},V)$, or, equivalently,
$$
   D_{u}H(\l,u)\in \Phi_{0}(U,V)\;\;\hbox{for all}\;\; (\l,u)\in [a,b]\times \mathcal{O}.
$$
Such a homotopy $H$ is said to be \emph{orientable} if
\(D_{u}H:[a,b]\times \mathcal{O}\to \Phi_{0}(U,V)\) is orientable.
We denote by $\varepsilon_{\l}$ the restriction of the orientation to the time slice $\l$, i.e.
\begin{equation}
	\label{ii.2}
	\varepsilon_{\l}: \mathscr{R}_{D_u H_{\l}}\longrightarrow \mathbb{Z}_{2},
	\qquad \varepsilon_{\l}(x):=\varepsilon(\l,x), \quad \l\in[a,b],
\end{equation}
where $H_\l\equiv H(\l,\cdot)$. Then, for any given open and bounded subset $\Omega\subset U$ with $\overline{\Omega}\subset\mathcal{O}$, the triple $(H,\Omega,\varepsilon)$ is called a \emph{Fredholm $\mathcal{O}$-admissible homotopy} if
\begin{enumerate}
	\item $H\in \mathscr{F}^{1}_{1}([a,b]\times \mathcal{O},V)$ is orientable with orientation
	$\varepsilon:\mathscr{R}_{D_{u}H}\to \mathbb{Z}_{2}$;
	\item $H$ is proper in $[a,b]\times\overline{\Omega}$;
	\item $0\notin H([a,b]\times \partial\Omega)$.
\end{enumerate}
The set of these homotopies is denoted by $\mathscr{H}(\mathcal{O})$.
\par
Finally, for any given $f:\mathcal{O}\to V$ of class $\mathcal{C}^{r}$, a point $u\in \mathcal{O}$ is
said to be a \emph{regular point} of $f$ if $Df(u)$ is surjective, which entails $Df(u)\in GL(U,V)$ if $f$ is $\mc{C}^{r}$-Fredholm of index zero. Subsequently, the set of regular points of $f$ is denoted by $\ms{R}_f\equiv \mathscr{R}_{Df}$.
Similarly, for every closed or open subset $\mathscr{O}\subset\mathcal{O}$, a point $v\in V$ is said to be a
\emph{regular value} of $f:\mathscr{O}\to V$ if $f^{-1}(v)\cap\mathscr{O}$ consists of regular points of $f$. The set of regular values of $f$ is denoted by $\mathscr{RV}_{f}(\mathscr{O})$, i.e.
$$
    \mathscr{RV}_{f}(\mathscr{O}):=\{v\in V\;:\; f^{-1}(v)\cap\mathscr{O}\subset \ms{R}_f\}.
$$
We already have all necessary ingredients to introduce the degree. Let $(f,\Omega,\varepsilon)\in \mathscr{A}(\mathcal{O})$ with $0\in \mathscr{RV}_{f}(\Omega)$. Then,
$$
   f^{-1}(0)\cap \Omega = f^{-1}(0)\cap \overline{\Omega}
$$
is finite, possibly empty, and the \emph{degree} is defined as
\[
  \deg(f,\Omega,\varepsilon):=\sum_{u\in f^{-1}(0)\cap\Omega}\varepsilon(u).
\]
If $f^{-1}(0)\cap \Omega=\emptyset$, we simply set $\deg(f,\Omega,\varepsilon):=0$. Lastly, if $0\notin \mathscr{RV}_{f}(\Omega)$, then we define
\[
\deg(f,\Omega,\varepsilon):=\deg(f-v,\Omega,\varepsilon),
\]
where $v\in V$ is a regular value sufficiently close to $0$, whose existence is guaranteed
by the Quinn--Sard--Smale theorem (see \cite{Sa} \cite{Sm} and \cite{QS}). Since $Df=D(f-v)$, the orientation $\varepsilon$ is the same for $f$ and $f-v$. As the definition of the degree is independent of the choice of the open set $\mathcal{O}$,  the degree is well defined as an integer-valued map on the quotient
$\mathscr{A}/\sim$, where
$$
\mathscr{A}:=\bigcup_{\mathcal{O}\subset U \text{ open}} \mathscr{A}(\mathcal{O}),
$$
and the equivalence relation $\sim$ identifies two admissible triples, $(f_1,\O_1,\e_1)$ and $(f_2,\O_2,\e_2)$, as soon as the next three properties hold: (i) $\Omega_{1}=\Omega_{2}=\Omega$, (ii) $f_{1}(u)=f_{2}(u)$ for all $u\in \overline{\Omega}$,
(iii) $\varepsilon_{1}(u)=\varepsilon_{2}(u)$ for all $u\in \mathscr{R}_{f_{1}}\cap\Omega=\mathscr{R}_{f_{2}}\cap\Omega$. The next theorem axiomatizes the degree for Fredholm operators.

\begin{theorem}
\label{th2.1}
There exists a unique integer-valued map
$$
   \deg:\mathcal{A}\equiv \mathscr{A}/\sim\;\longrightarrow \mathbb{Z}
$$
satisfying the following properties:
\begin{enumerate}
\item[{\rm (N)}] \textbf{Normalization:} Suppose $L\in GL(U,V)$ has orientation $\varepsilon$
and $\Omega$ is  an open and bounded subset of $U$ with $0\in\Omega$. Then,
$$
		\deg(L,\Omega,\varepsilon)=\varepsilon(0).
$$
\item[{\rm (A)}] \textbf{Additivity:} Suppose $(f,\Omega,\varepsilon)\in\mathcal{A}$ and $\Omega_{1},\Omega_{2}\subset \Omega$ are disjoint open subsets with $0\notin f\big(\Omega\setminus(\Omega_{1}\uplus\Omega_{2})\big)$. Then,
\begin{equation}
			\label{ii.3}
			\deg(f,\Omega,\varepsilon)
			=\deg(f,\Omega_{1},\varepsilon)
			+\deg(f,\Omega_{2},\varepsilon).
\end{equation}
		\item[{\rm (H)}] \textbf{Homotopy Invariance:}
		Suppose $(H,\Omega,\varepsilon)\in \mathscr{H}(\mathcal{O})$ is a Fredholm $\mathcal{O}$-admissible homotopy. Then,
\begin{equation}
\label{ii.4}
			\deg(H(a,\cdot),\Omega,\varepsilon_{a})
			=\deg(H(b,\cdot),\Omega,\varepsilon_{b}).
\end{equation}
\end{enumerate}
\end{theorem}

Axiom (A) packages the additivity, the excision and the existence properties of the degree, as discussed by the authors in \cite{LGS22}. The existence of the degree was established by Fitzpatrick, Pejsachowicz and Rabier in \cite{FPR} for $\mc{C}^2$ mappings based on the concept of orientability introduced by Fitzpatrick and Pejsachowicz in \cite{FP91,FP93},  and later generalized to $\mc{C}^1$ mappings by Pejsachowicz and Rabier in \cite{PR}. The uniqueness was proven by the authors in \cite{LGS22} based on the generalized Leray--Schauder formula of the authors in \cite{LGS20} through the generalized algebraic multiplicity of Esquinas and L\'{o}pez-G\'{o}mez \cite{ELG,Es,LG01}. The monograph of L\'{o}pez-G\'{o}mez and Mora-Corral \cite{LGM} establishes the uniqueness of the generalized algebraic multiplicity.

\subsection{Generalized Homotopy Invariance} In this section we deliver a generalized homotopy invariance property appropriate for the requirements of this paper.  For every subset $\O\subset \mathbb{R}\times U$ and any $\l\in \mathbb{R}$, we are denoting
$$
\O_\l:=\{u\in U\,:\;(\l,u)\in\O\}.
$$
For every  open bounded subset $\O\subset [a,b]\times U$, any open subset $\mc{O}\subset U$ such that $\overline{\O}\subset [a,b]\times \mc{O}$, and any homotopy $H:[a,b]\times \mc{O}\to V$, it is said that $(H,\O,\varepsilon)$ is a \textit{generalized  $\mc{O}$--admissible homotopy} if the following three properties
hold:
\begin{enumerate}
	\item $H\in\mathscr{F}^{1}_{1}([a,b]\times \mc{O},V)$ is \textit{orientable} with orientation $\varepsilon:\mathscr{R}_{D_u H}\to\Z_{2}$;
	\item $H$ is proper on $\overline{\O}$;
	\item $0\notin H(\partial\O)$.
\end{enumerate}
The class of the generalized $\mc{O}$--admissible homotopies is denoted by $\mathscr{G}(\mc{O})$ in this paper.

\begin{theorem}
\label{thb.2}
Let $(H,\O,\varepsilon)\in \mathscr{G}(\mc{O})$ be a generalized $\mc{O}$--admissible homotopy. Then,
\begin{equation*}
		\deg(H_{a},\Omega_{a},\varepsilon_{a})=\deg(H_{b},\Omega_{b},\varepsilon_{b}).
\end{equation*}
\end{theorem}

\begin{proof}
Without loss of generality, we may assume that $[a,b] = [0,1]$. We claim that, for every $\l_{0}\in[0,1]$,  there exists $\eta>0$ such that:
\begin{equation}
\label{e1}
	H^{-1}_{\l}(0)\cap \O_{\l}=H^{-1}_{\l}(0)\cap \O_{\l_{0}} \;\;\hbox{if}\;\; |\l-\l_0|\leq \eta.
\end{equation}
Suppose \eqref{e1} fails for all $\eta>0$. Then, there exists a sequence $\{(\l_{n},u_{n})\}_{n\in\mathbb{N}}$ in $H^{-1}(0)\cap \O$ such that
$$
   \lim_{n\to+\infty}\l_{n}=\l_{0}\;\;\hbox{and}\;\; u_{n}\in\O_{\l_{n}}\backslash\O_{\l_{0}} \;\;\hbox{for all}
   \;\; n\geq 1.
$$
Since $H^{-1}(0)\cap \overline{\O}$ is compact, without loss of generality, we can assume that
$$
	\lim_{n\to+\infty}(\l_{n},u_{n})=(\l_{0},u_{0})\in H^{-1}(0)\cap \overline{\O}.
$$
Since $0\notin H(\partial\O)$, necessarily $(\l_{0},u_{0})\in H^{-1}(0)\cap \O$. Thus, $u_{0}\in H_{\l_{0}}^{-1}(0)\cap \O_{\l_{0}}$. In particular $\O_{\l_{0}}\neq \emptyset$. But this contradicts the fact that $u_{n}\in\O_{\l_{n}}\backslash\O_{\l_{0}}$ for all  $n\geq 1$. So, \eqref{e1} also holds.
\par
By the compactness of $[0,1]$, there is some integer $m\in\mathbb{N}$ such that, setting
$$
	\l_{i}:=\frac{i}{m}, \qquad 0\leq i \leq m,
$$
we have that
$$
    H^{-1}_{\l}(0)\cap \O_{\l}=H^{-1}_{\l}(0)\cap \O_{\l_{i}}\;\;\hbox{for all}\;\; \l\in[\l_{i},\l_{i+1}],\;\;
    0\leq i \leq m-1.
$$
Thus, thanks to the excision property of the degree,
\begin{equation*}	
   \deg(H_{\l},\O_{\l},\varepsilon_{\l})=\deg(H_{\l},\O_{\l_{i}},\varepsilon_{\l})
   \;\;\hbox{for all}\;\; \l\in[\l_{i},\l_{i+1}],\;\;     0\leq i \leq m-1.
\end{equation*}
In particular,
\begin{equation}
\label{ii.6}	 \deg(H_{\l_{i+1}},\O_{\l_{i+1}},\varepsilon_{\l_{i+1}})=\deg(H_{\l_{i+1}},\O_{\l_{i}},\varepsilon_{\l_{i+1}}).
\end{equation}
Moreover, since $0\notin H([\l_{i},\l_{i+1}]\times\partial\O_{\l_{i}})$, by the homotopy invariance property,
\begin{equation*}
		 \deg(H_{\l_{i}},\O_{\l_{i}},\varepsilon_{\l_{i}})=\deg(H_{\l_{i+1}},\O_{\l_{i}},\varepsilon_{\l_{i+1}}).
\end{equation*}
Consequently, thanks to \eqref{ii.6}, we find that
\begin{equation*}
		 \deg(H_{\l_{i}},\O_{\l_{i}},\varepsilon_{\l_{i}})=\deg(H_{\l_{i+1}},\O_{\l_{i+1}},\varepsilon_{\l_{i+1}})
\;\; \hbox{for all}\;\; i\in\{0,...,m-1\}.
\end{equation*}
Therefore,
$$
    \deg(H_{0},\O_{0},\varepsilon_{0})=\deg(H_{\l_{0}},\O_{\l_{0}},\varepsilon_{\l_{0}})=
    \deg(H_{\l_{m}},\O_{\l_{m}},\varepsilon_{\l_{m}})=\deg(H_{1},\O_{1},\varepsilon_{1}),
$$
as claimed.
\end{proof}

\section{Topological continuation theorems}
\label{Se3}

\noindent In this section we deliver the main continuation theorems of this paper. Their proofs rely to a large extent on the existence of open isolating neighborhoods for the solution components of the underlying nonlinear equations, which is based on a well known result in the context of bifurcation theory, going back to
Whyburn \cite{Wh} and reading as follows.

\begin{lemma}
\label{le3.1}
	Let $(M,d)$ be a compact metric space and $A$ and $B$ two disjoint compact subsets of M. Then, either there exists a connected component of $M$ meeting both $A$ and $B$, or
$$
   M = M_{A} \uplus M_{B},
$$
where $M_{A}$ and $M_{B}$ are disjoint compact subsets of $M$ containing $A$ and $B$, respectively.
\end{lemma}

The next continuation theorem, going back to Theorem 3.8 of \cite{LGS24}, is a substantial generalization
of a previous result of Mawhin \cite{Ma} in the context of the Leray--Schauder degree, which goes back to
Leray and Schauder \cite{LeS}. For the sake of completeness, we include here a self-contained proof, since the complete details were not provided in \cite{LGS24}.

\begin{theorem}
\label{th3.1}
Let $(H,\O,\varepsilon)\in\mathscr{H}(\mc{O})$ be a Fredholm $\mc{O}$--admissible homotopy with
$$
   \deg(H_{a},\Omega,\varepsilon_a)\neq 0.
$$
Then, there exists a connected component $\mathscr{C}\subset H^{-1}(0)\cap([a,b]\times\O)$ that connects $\{a\}\times \Omega$ with $\{b\}\times\Omega$.
\end{theorem}

\begin{proof}
	Since $\deg(H_{a},\Omega,\varepsilon_a)\neq 0$, by the existence property of the degree,  $H_{a}^{-1}(0)\cap\O\neq \emptyset$.  Let $\mathscr{D}$ be the disjoint union of the connected components, $\mathscr{C}$, of the set
$$
   \mathscr{S}:=H^{-1}(0)\cap([a,b]\times\O)=H^{-1}(0)\cap ([a,b]\times \bar{\O})
$$
such that
$$
    \mathscr{C}\cap H_{a}^{-1}(0)\neq \emptyset.
$$
Suppose that $\mathscr{D}$ intersects $\{b\}\times\Omega$. Then, by choosing $\mathscr{C}$ as  one of the connected components of $\mathscr{D}$, the proof is completed. So, suppose that $\mathscr{D}$ does not intersect $\{b\}\times\Omega$.
\par
The properness hypothesis on $H$ guarantees that $\mathscr{S}$ is compact. Thus, $\mathscr{D}$ is also compact. Moreover, since $0\notin H([a,b]\times\partial\O)$, we also have that
$$
   \mathscr{D}\cap ([a,b]\times\partial\O)=\emptyset.
$$
Hence, since $\mathscr{D}$ is compact and $[a,b]\times\partial\O$ is closed,
$$
   d:=\text{dist}(\mathscr{D},[a,b]\times\partial\O)>0.
$$
Subsequently, for any given $\d\in (0,d)$, we consider the open subset
$$
	\mc{N}\equiv\mc{N}(\d):=\left\{(\l,u)\in [a,b]
     \times \O : \text{dist}((\l,u),\mathscr{D})<\d\right\}\subset [a,b]\times \O.
$$
By the choice of $d$,  $\overline{\mc{N}}\subset [a,b]\times \O$. Now, consider the compact set $M:=\overline{\mc{N}}\cap \mathscr{S}$ and its subsets
$$
	A:= \mathscr{D}\subset M, \qquad B:=\partial\mc{N}\cap \mathscr{S}\subset M.
$$
By construction, $A\cap B=\emptyset$. Moreover, since $M\subset \mathscr{S}$ and $\mathscr{D}$ is a union of connected components of $\mathscr{S}$, there cannot exist any connected subset of $M$ joining $A$ and $B$. Therefore, by Lemma \ref{le3.1}, there are two disjoint compact subsets $M_{A}, M_{B}\subset M$ such that $A\subset M_{A}$, $B\subset M_{B}$ and $M=M_{A}\uplus M_{B}$. Note that
$$
   M_{A}\cap\partial\mc{N}=M_{A}\cap\mathscr{S}\cap\partial\mc{N}=M_{A}\cap B=\emptyset.
$$
Thus, as $M_{A}$ and $M_{B}$ are compact and $\partial\mc{N}$ is closed, we find that
$$
   \text{dist}(M_{A},M_{B})>0\;\;\hbox{and}\;\; \text{dist}(M_{A},\partial\mc{N})>0.
$$
Hence,
$$
	\eta:=\min\left\{\text{dist}(M_{A},M_{B}),\text{dist}(M_{A},\partial\mc{N})\right\}>0
$$
and we can consider the open subset
$$
	\mc{V}\equiv\mc{V}(\d):=\left\{(\l,u)\in \mc{N} : \text{dist}((\l,u), M_{A})<
     \tfrac{\eta}{2}\right\}\subset \mc{N}.
$$
Since $\mathscr{D}=A\subset M_A$, we have that $\mathscr{D}\subset \mc{V}$. Moreover, by definition of
$\mc{V}$, $M_{A}\cap \partial\mc{V}=\emptyset$ and, since $\frac{\eta}{2} < \text{dist}(M_{A},M_{B})$, necessarily $M_{B}\cap \partial\mc{V}=\emptyset$. Therefore,
\begin{equation}
\label{iii.1}
   M\cap \partial\mc{V}=(M_A\cap \partial\mc{V})\cup(M_B\cap \partial\mc{V})=\emptyset.
\end{equation}
On the other hand, since $\mc{V}\subset \mc{N}$, we have that $\partial\mc{V}\subset\overline{\mc{N}}$. Thus,
\eqref{iii.1} implies that
\begin{equation}
\label{iii.2}
	H^{-1}(0)\cap \partial\mc{V}=\mathscr{S}\cap \partial\mc{V}=\mathscr{S}\cap \overline{\mc{N}}\cap \partial\mc{V}=M\cap \partial\mc{V}=\emptyset.
\end{equation}
As we are assuming that $\mathscr{D}$ does not intersect $\{b\}\times\Omega$ and it is compact, there exists
$\l_0\in [a,b)$ such that $\mathcal{P}_{\l}(\mathscr{D})=[a,\l_{0}]$, where $\mathcal{P}_{\l}$ stands for the
$\l$-projection operator
$$
   \mathcal{P}_{\l}:[a,b]\times \O\to [a,b], \quad (\l,u)\mapsto \l.
$$
Since  $H_{a}^{-1}(0)\cap\O = \mathscr{D}_{a}\subset\mathcal{V}_{a}$, by the excision property of the degree, it is apparent that
$$
   \deg(H_{a},\Omega,\varepsilon_{a})=\deg(H_{a},\mathcal{V}_{a},\varepsilon_{a}).
$$
Thus, thanks to Theorem \ref{thb.2},
$$
	0\neq \deg(H_{a},\Omega,\varepsilon_{a})=\deg(H_{\l},\mathcal{V}_{\l},\varepsilon_{\l})\quad \hbox{for all}\;\; \l\in [a,b].
$$
Therefore, by the fundamental property of the degree, $\mathcal{V}_\l\neq \emptyset$ for all $\l\in (\l_0,b]$,  which is impossible for sufficiently small $\d>0$. This ends the proof.
\end{proof}

To state and prove the main continuation theorem of this paper, we need to introduce some basic notions and notations. First, we introduce the notion of  \emph{oriented index of a zero}. For any given $(f,\Omega,\varepsilon)\in \mathscr{A}(\mathcal{O})$, suppose that $u\in f^{-1}(0)\cap \Omega$ is an isolated zero of $f$. Then, for sufficiently small $\delta>0$, say $\d\leq \d_0$, we have that $B_{\delta}(u)\subset \Omega$ and
$$
(B_{\delta}(u)\setminus\{u\})\cap f^{-1}(0)=\emptyset.
$$
We define the oriented index of $u$ by
$$
i(f,u,\varepsilon):=\deg(f,B_{\delta}(u),\varepsilon),\qquad \d\leq \d_0.
$$
By the excision property of the degree, $\deg(f,B_{\delta}(u),\varepsilon)$ is independent of $\d\in (0,\d_0]$. Thus, $i(f,u,\varepsilon)$ is well defined.
\par
Next, we consider $\mathcal{C}^{1}$ operators
$$
 \mathfrak{F}:\mathbb{R}\times \mathcal{O}\longrightarrow V
$$
satisfying the following assumptions:
\begin{enumerate}
	\item[(F1)] $\mathfrak{F}\in\mathscr{F}^{1}_{1}(\mathbb{R}\times \mathcal{O},V)$ is \emph{orientable}, with orientation 	$\varepsilon:\mathscr{R}_{D_{u}\mathfrak{F}}\to\mathbb{Z}_{2}$;
	\item[(F2)] $\mathfrak{F}$ is proper on closed and bounded subsets of $\overline{\mathcal{U}}$, where $\mathcal{U}\neq \emptyset$ is any open subset of $\mathbb{R}\times \mathcal{O}$ such that
$\overline {\mathcal{U}}\subset \mathbb{R}\times \mathcal{O}$.
\end{enumerate}
Set
$$
\mathscr{S}:=\mathfrak{F}^{-1}(0)\cap\mathcal{U}.
$$
Clearly, $\mathscr{S}$ is closed in $\mathcal{U}$. Subsequently, for every $\lambda_{0}\in \mathbb{R}$, we introduce the sets
\begin{align*}
\mathcal{U}_{\lambda_{0},c}^{+} & :=\mathcal{U}\cap([\lambda_{0},+\infty)\times \mathcal{O}),\quad
\mathscr{S}_{\lambda_{0},c}^{+}  :=\mathscr{S}\cap\mathcal{U}_{\lambda_{0},c}^{+}
=\{(\lambda,u)\in\mathscr{S}:\lambda\geq \lambda_{0}\}, \\
\mathcal{U}_{\lambda_{0},c}^{-} & :=\mathcal{U}\cap((-\infty,\lambda_{0}]\times \mathcal{O}),\quad
\mathscr{S}_{\lambda_{0},c}^{-} :=\mathscr{S}\cap\mathcal{U}_{\lambda_{0},c}^{-}
=\{(\lambda,u)\in\mathscr{S}:\lambda\leq \lambda_{0}\},
\end{align*}
where the subindex \lq\lq $c$\rq\rq\!  makes reference to the fact that the intervals $[\l_0,+\infty)$ and
$(-\infty,\l_0]$ are closed. It is easily seen that $\mathscr{S}^{\pm}_{\lambda_{0},c}$ is closed in both $\mathcal{U}^{\pm}_{\lambda_{0},c}$, respectively,  and in $\mathscr{S}$, and, by definition,
\[
\mathscr{S}^{+}_{\lambda_{0},c}\cap \mathscr{S}^{-}_{\lambda_{0},c}=\mathfrak{F}^{-1}_{\lambda_{0}}(0)\cap \mc{U}.
\]
Moreover, by assumption (F2), each of the sets $\mathscr{S}^{\pm}_{\lambda_{0},c}$ is locally compact.
The main continuation theorem of this section can be stated as follows. It is a substantial generalization of
Theorem 1 of Santos, Cintra and Ramos \cite{SCR} --- based on Theorem 3.5 of Rabinowitz \cite{Ra71b}.

\begin{theorem}
\label{th3.2}
	Let $\mf{F}:\R\times \mc{O}\to V$ be a $\mc{C}^{1}$-operator satisfying {\rm{(F1)--(F2)}}, and suppose that $u_{0}\in\mc{U}_{\l_0}$ is an isolated zero of $\mf{F}_{\l_{0}}$ such that  $i(\mf{F}_{\l_{0}},u_{0},\varepsilon_{\l_{0}})\neq 0$.  Then, there exist two connected components $\mathscr{C}^{\pm}$ of $\mathscr{S}_{\l_{0},c}^{\pm}$, respectively,  such that $(\l_{0},u_{0})\in \mathscr{C}^{\pm}$ for which one of the following non-excluding alternatives holds:
	\begin{enumerate}
		\item[{\emph{(i)}}] $\mathscr{C}^{\pm}$ is unbounded.
		\item[{\emph{(ii)}}] $\mathscr{C}^{\pm}\cap \partial\mc{U}\neq \emptyset$.
		\item[{\emph{(iii)}}] There exists $u_{1}\in\mc{U}_{\l_{0}}$, $u_{1}\neq u_{0}$, such that $(\l_{0},u_{1})\in \mathscr{C}^{\pm}$.
	\end{enumerate}
Moreover, if $\mathscr{C}^{\pm}$ is bounded, $\mathscr{C}^{\pm}\cap \partial\mc{U}=\emptyset$, and
$\mf{F}^{-1}_{\l_{0}}(0)\cap \mathscr{C}^{\pm}_{\l_{0}}$ is discrete, then
\begin{equation}
\label{iii.3}
	\sum_{u\in \mf{F}^{-1}_{\l_{0}}(0)\cap \mathscr{C}^{\pm}_{\l_{0}}}i(\mf{F}_{\l_{0}},u,\varepsilon_{\l_{0}})=0.
\end{equation}
In particular, there are $2\nu$, $\nu\geq 1$, points $u\in \mf{F}^{-1}_{\l_{0}}(0)\cap\mathscr{C}^{\pm}_{\l_{0}}$ with $i(\mf{F}_{\l_{0}},u,\varepsilon_{\l_{0}})\neq 0$.
\end{theorem}

\begin{proof}
By symmetry, it suffices to show the existence of $\mathscr{C}^{+}$. Since $u_{0}\in\mc{U}_{\l_0}$ is an isolated zero of $\mf{F}_{\l_{0}}$ such that  $i(\mf{F}_{\l_{0}},u_{0},\varepsilon_{\l_{0}})\neq 0$, it is easily seen (e.g. from Theorem \ref{th3.1}) that there exists a (unique) connected component $\mathscr{C}^{+}$ of $\mathscr{S}^{+}_{\l_{0},c}$ such that $(\l_{0},u_{0})\in \mathscr{C}^{+}$. Arguing by contradiction,
assume that $\mathscr{C}^{+}$ does not satisfy any of the alternatives (i)--(iii). Then,
$$
      \mathscr{C}^{+}\;\;\hbox{is bounded,}\quad \mathscr{C}^{+}\cap\partial\mc{U}=\emptyset,\quad
      \hbox{and}\;\; \mathcal{U}_{\l_0}\cap \mathscr{C}^{+}_{\l_{0}}=\{(\l_{0},u_{0})\}.
$$
Thus, since $\mathscr{C}^{+}\subset \mc{U}$ is closed and bounded, by hypothesis (F2), $\mathscr{C}^{+}$ is compact. Hence,
$$
	d:=\text{dist}(\mathscr{C}^{+},\partial\mc{U}^{+}_{\l_{0},c})>0.
$$
Since $u_{0}\in \mc{U}_{\l_{0}}$ is an isolated zero of $\mf{F}_{\l_{0}}$, there exists $\tau>0$ such that
$$
   \mf{F}^{-1}_{\l_{0}}(0)\cap B_{\tau}(u_{0})=\{u_{0}\}.
$$
Pick any $\d$ such that $0<\d<\min\{d,\tau\}$ and consider the open  set
\begin{equation}
\label{iii.4}
	\mc{N}\equiv\mc{N}(\d):=\left\{(\l,u)\in \mc{U}^{+}_{\l_{0},c}: \text{dist}((\l,u),\mathscr{C}^{+})<\d\right\}.
\end{equation}
By construction, $\mc{N}$ is bounded  and
$$
    \overline{\mc{N}}\subset \mc{U}^{+}_{\l_{0},c},\quad
    \mc{N}_{\l_{0}}\cap \mf{F}_{\l_{0}}^{-1}(0)=\{u_{0}\}.
$$
Suppose that
\begin{equation}
\label{iii.5}
  \partial\mc{N}\cap \mathscr{S}^{+}_{\l_{0},c}=\emptyset.
\end{equation}
Then, it follows from Theorem \ref{thb.2} that, as soon as $\mc{N}_\l =\emptyset$, which occurs for sufficiently large $\l>\l_0$,
$$
  0  \neq i(\mf{F}_{\l_{0}},u_{0},\varepsilon_{\l_{0}})  = \deg(\mf{F}_{\l_{0}},B_{\delta}(u_0),\varepsilon_{\l_{0}})
   = \deg(\mf{F}_\l,\mc{N}_\l,\varepsilon_{\l})=0,
$$
which is impossible. Therefore, \eqref{iii.5} fails.
\par
Subsequently, we consider the set of solutions in $\overline{\mc{N}}$,
$$
    M:=\overline{\mc{N}}\cap \mathscr{S}^{+}_{\l_{0},c}.
$$
Since $M\subset \mf{F}^{-1}(0)$ is bounded and closed, by (F2), it is necessarily compact. Next, we consider  the nonempty subsets
$$
	A:= \mathscr{C}^{+}\subset M, \quad B:=\partial\mc{N}\cap \mathscr{S}^{+}_{\l_{0},c}\subset M.
$$
	By construction, $A\cap B=\emptyset$. Moreover, since $M\subset \mathscr{S}^{+}_{\l_{0},c}$ and $\mathscr{C}^{+}$ is a connected component of $\mathscr{S}^{+}_{\l_{0},c}$, there cannot exist a connected subset of $M$ joining $A$ and $B$. Thus, by Lemma \ref{le3.1}, there exist two disjoint compact subsets $M_{A}, M_{B}\subset M$ such that $A\subset M_{A}$, $B\subset M_{B}$ and $M=M_{A}\uplus M_{B}$. Note that
$$
	M_{A}\cap\partial\mc{N}=M_{A}\cap\mathscr{S}^{+}_{\l_{0},c}\cap\partial\mc{N}=M_{A}\cap B=\emptyset.
$$
Hence, as $M_{A}$ and $M_{B}$ are compact and $\partial\mc{N}$ is closed, it follows that
$$
 \text{dist}(M_{A},M_{B})>0, \quad \text{dist}(M_{A},\partial\mc{N})>0.
$$
Set
$$
	\eta:=\min\left\{\text{dist}(M_{A},M_{B}),\text{dist}(M_{A},\partial\mc{N})\right\}
$$
and consider the open neighborhood of $M_A$ defined by
\begin{equation}
\label{iii.6}
	\mc{V}\equiv\mc{V}(\d):=\left\{(\l,u)\in \mc{N} : \text{dist}((\l,u), M_{A})<\dfrac{\eta}{2}\right\}\subset \mc{N}.
\end{equation}
Then, since $\mc{V}\subset \mc{N}$ and $\mc{N}$ is bounded, $\mc{V}$ is bounded. Moreover, since
$\mathscr{C}^+\subset M_A$,  we have that $\mathscr{C}^{+}\subset \mc{V}$. Now, we claim that
\begin{equation}
\label{iii.7}
    \mf{F}^{-1}(0)\cap \partial\mc{V}=\emptyset.
\end{equation}
Indeed, since $\eta < \text{dist}(M_{A},M_{B})$, necessarily $M_{B}\cap \partial\mc{V}=\emptyset$. Thus,
since $M_A\subset \mc{V}$, we find that $M\cap \partial\mc{V}=\emptyset$. Consequently, since
$\partial\mc{V}\subset\overline{\mc{N}}$, it becomes apparent that
$$
	\mf{F}^{-1}(0)\cap \partial\mc{V}=\mathscr{S}^{+}_{\l_{0},c}\cap \partial\mc{V}=\mathscr{S}^{+}_{\l_{0},c}\cap \overline{\mc{N}}\cap \partial\mc{V}=M\cap \partial\mc{V}=\emptyset.
$$
Moreover, $\mf{F}_{\l_{0}}^{-1}(0)\cap\mc{V}_{\l_{0}}=\{u_{0}\}$ because $u_{0}\in\mc{V}_{\l_{0}}\subset \mc{N}_{\l_{0}}$ and $\mf{F}_{\l_{0}}^{-1}(0)\cap\mc{N}_{\l_{0}}=\{u_{0}\}$. Furthermore,  since $\mc{V}$ is bounded, there exists $\l_{\ast}>\l_{0}$ such that $\mc{V}_{\l}=\emptyset$ for all $\l \geq \l_{\ast}$. Therefore, the triple $(\mf{F},\mc{V},\varepsilon)$ is a generalized Fredholm $\mc{O}$--admissible homotopy for the restriction $\mf{F}:[\l_{0},\l_{\ast}]\times \mc{O}\to V$. Indeed, since $\overline{\mc{V}}$ is closed and bounded, $\mf{F}$ is proper on $\overline{\mc{V}}$ and, by construction, $0\notin \mf{F}(\partial\mc{V})$. By Theorem \ref{thb.2},
$$
	\deg(\mf{F}_{\l_{0}},\mc{V}_{\l_{0}},\varepsilon_{\l_{0}})=
\deg(\mf{F}_{\l_{\ast}},\mc{V}_{\l_{\ast}},\varepsilon_{\l_{\ast}}).
$$
Since $\mc{V}_{\l_{\ast}}=\emptyset$, necessarily $\deg(\mf{F}_{\l_{\ast}},\mc{V}_{\l_{\ast}},\varepsilon_{\l_{\ast}})=0$. On the other hand, as $\mf{F}_{\l_{0}}^{-1}(0)\cap\mc{V}_{\l_{0}}=\{u_{0}\}$, by the excision property, we find that, for sufficiently small $\eta>0$,
$$
0=\deg(\mf{F}_{\l_{0}},\mc{V}_{\l_{0}},\varepsilon_{\l_{0}})=
\deg(\mf{F}_{\l_{0}},B_{\eta}(u_{0}),\varepsilon_{\l_{0}})=
i(\mf{F}_{\l_{0}},u_0,\varepsilon_{\l_{0}})\neq 0.
$$
This contradiction concludes the proof of the first part of the theorem.
\par
Suppose now that alternatives (i) and (ii) do not hold and that $\mf{F}^{-1}_{\l_{0}}(0)\cap \mathscr{C}^{+}_{\l_{0}}$ is discrete. Since $\mathscr{C}^{+}$ is compact, this implies that $\mf{F}^{-1}_{\l_{0}}(0)\cap\mathscr{C}^{+}_{\l_{0}}$ is finite. Thus, there exist $m\in\N$ and $u_{0}^{i}\in \mc{U}_{\l_{0}}$, $j\in\{1,\dots,m\}$, such that
$$
    \mf{F}^{-1}_{\l_{0}}(0)\cap\mathscr{C}^{+}_{\l_{0}}=\{(\l_{0},u_{0}^{j})\}_{j=1}^{m}.
$$
Moreover, for every $j\in\{1,\dots,m\}$, there exists $\tau_{j}>0$ such that $\mf{F}^{-1}_{\l_{0}}(0)\cap B_{\tau_{j}}(u_{0}^{j})=\{u_{0}^{j}\}$. Choose a $\delta$ satisfying
$$
   0<\d<\min\{d, \tau_{1},\dots,\tau_{m}\},
$$
and define $\mc{N}\equiv \mc{N}(\d)$ and $\mc{V}\equiv\mc{V}(\d)$ as in \eqref{iii.4} and \eqref{iii.6}, respectively. Then, $\mc{V}$ is bounded, $\mathscr{C}^{+}\subset \mc{V}$, $\mf{F}^{-1}(0)\cap \partial\mc{V}=\emptyset$ and $\mf{F}^{-1}_{\l_{0}}(0)\cap \mc{V}_{\l_{0}}=\{u_{0}^{j}\}_{j=1}^{m}$. Again, since $\mc{V}$ is bounded, there exists $\l_{\ast}>\l_{0}$ such that $\mc{V}_{\l}=\emptyset$ for all $\l \geq \l_{\ast}$. Thus, by Theorem \ref{thb.2},
$$
    \deg(\mf{F}_{\l_{0}},\mc{V}_{\l_{0}},\varepsilon_{\l_{0}})=\deg(\mf{F}_{\l_{\ast}},\mc{V}_{\l_{\ast}},\varepsilon_{\l_{\ast}})=0.
$$
On the other hand, by the additivity and the excision property of the degree,
$$
    \deg(\mf{F}_{\l_{0}},\mc{V}_{\l_{0}},\varepsilon_{\l_{0}})=
    \sum_{j=1}^{m}\deg(\mf{F}_{\l_{0}},B_{\tau_{j}}(u_{0}^{j}),\varepsilon_{\l_{0}})=
    \sum_{j=1}^{m}i(\mf{F}_{\l_{0}},u_{0}^{j},\varepsilon_{\l_{0}}).
$$
This concludes the proof.
\end{proof}

As a byproduct of Theorem \ref{th3.2}, the next result holds.

\begin{corollary}
\label{co3.1}
Let $\mf{F}:\R\times \mc{O}\to V$ be a $\mc{C}^{1}$ operator satisfying {\rm{(F1)--(F2)}}. Suppose that
$$
   \mf{F}^{-1}_{\l_{0}}(0)\cap \mc{U}_{\l_{0}}=\{u_{0}\}\; \; \text{with} \;\; i(\mf{F}_{\l_{0}},u_{0},\varepsilon_{\l_{0}})\neq 0.
$$
Then, there are two connected components $\mathscr{C}^{\pm}\subset \mathscr{S}_{\l_{0},c}^{\pm}$ such that $(\l_{0},u_{0})\in \mathscr{C}^{\pm}$, and either $\mathscr{C}^{\pm}$ is unbounded, or
$\mathscr{C}^{\pm}\cap \partial\mc{U}\neq \emptyset$.
\end{corollary}

Theorem \ref{th3.2} also extends the next continuation theorem for Fredholm operators of Dai and Zhang \cite{DZ}.

\begin{corollary}
\label{co3.2}
	Let $\mf{F}:\R\times U\to V$ be a $\mc{C}^{1}$ operator satisfying {\rm{(F1)--(F2)}}. Suppose that $\mf{F}^{-1}_{\l_{0}}(0)=\{u_{0}\}$ and $D_{u}\mf{F}(\l_{0},u_{0})\in GL(U,V)$. Then, there exist two unbounded connected components
	$$
	\mathscr{C}^{+}\subset \{(\l,u)\in\mf{F}^{-1}(0) : \l\geq \l_{0}\}, \quad \mathscr{C}^{-}\subset \{(\l,u)\in \mf{F}^{-1}(0) : \l \leq \l_{0}\},
	$$ such that $(\l_{0},u_{0})\in \mathscr{C}^{\pm}$.
\end{corollary}

\begin{proof}
Note that $0\in\mathscr{RV}_{\mf{F}_{\l_{0}}}(U)$ since $\mf{F}^{-1}_{\l_{0}}(0)=\{u_{0}\}$ and $D_{u}\mf{F}(\l_{0},u_{0})\in GL(U,V)$. Thus, the definition of the degree for regular points yields
$$
  i(\mf{F}_{\l_{0}},u_{0},\varepsilon_{\l_{0}})=
  \deg(\mf{F}_{\l_{0}},B_{\tau}(u_{0}),\varepsilon_{\l_{0}})=\varepsilon(u_{0})\neq 0.
$$
Now, the result holds from Corollary \ref{co3.1} for the case $\mc{O}=U$ and $\mc{U}=\R\times U$.
\end{proof}

\section{Counterexample without compactness assumption}

\noindent The proof of Theorem \ref{th3.1} is based upon the construction of appropriate \emph{open isolating neighborhoods} for each of the components $\mathscr{C}^\pm$. Roughly speaking, an open isolating neighborhood is an open set, $\mathscr{N}^\pm$,  such that
$$
   \mathscr{C}^\pm\subset \mathscr{N}^\pm\;\; \hbox{and}\;\; \mathscr{S}^\pm \cap \p \mathscr{N}^\pm=\emptyset.
$$
In this section we prove that, in the absence of compactness assumptions,
even though the connected components are closed and bounded, open
isolating neighborhoods may fail to exist. We first fix the concept of open isolating neighborhood.

\begin{definition}
\label{de4.1}
Let $U$ be a real Banach space and consider a bounded closed subset $\mathscr{B}\subset U$. Let $\mathscr{C}$ be a connected component of $\mathscr{B}$ and $\d>0$. An open subset $\mathcal{V}\subset U$ is called an \textit{open isolating neighborhood of size $\d$} of $\mathscr{C}$, with respect to $\mathscr{B}$, if
\begin{enumerate}
	\item $\mathscr{C}\subset\mc{V}\subset \mc{N}(\d):=\left\{u\in U : \text{dist}(u,\mathscr{C})<\d\right\}$.
	\item $\mathscr{B}\cap \partial\mc{V}=\emptyset$.
\end{enumerate}
\end{definition}

The next result shows that small isolating neighborhoods may fail to exist.

\begin{theorem}
\label{pr4.1}
Let $U$ be an infinite-dimensional real Banach space. Then, there exist a closed and bounded subset $\mathscr{B}\subset U$ with exactly two connected components $\mathscr{B}=\mathscr{C}_{1}\uplus \mathscr{C}_{2}$, and a constant $\delta^{\ast}>0$ such that
neither $\mathscr{C}_1$ nor $\mathscr{C}_2$ can admit an open isolating neighborhood of size $\d$ with respect to $\mathscr{B}$ for any $\delta\in (0,\delta^{\ast})$.
\end{theorem}

\begin{proof}
By a celebrated theorem of Riesz \cite{Ri} on nearly orthogonal elements (see, e.g., Yosida \cite[p. 84]{Yo}), there exists a sequence $\{u_{n}\}_{n\in\N}\subset U$ such that
\begin{equation}
\label{iv.1}
	\|u_n\|=1 \;\; \text{and} \;\; \|u_{n}-u_{m}\|\geq \tfrac{1}{2} \;\; \text{for all} \;\; n,m\in \N, \;\; n\neq m.
\end{equation}
Now, consider the subsets of the closed unit ball
$$
    \mathscr{C}_{1}=\mathbb{S}(U):=\left\{u\in U  : \|u\|=1\right\},
$$
and
$$
	\mathscr{C}_{2}:=\biguplus_{n\in\N}\left\{t u_{n} : 0<t\leq t_{n}\right\}\uplus \{0\}, \quad t_{n}:=1-\tfrac{1}{n+1}.
$$
Obviously, $\mathscr{C}_{1}$ is closed, bounded, and connected, and $\mathscr{C}_{2}$ is bounded and connected. We claim that $\mathscr{C}_{2}$ is also closed. Indeed, let $\{s_{m}u_{n(m)}\}_{m\in\N}\subset \mathscr{C}_{2}$ be a Cauchy sequence. Then, by \eqref{iv.1}, for every $\varepsilon>0$, there exists $m_0=m_0(\e)\in \N$ such that
\begin{equation}
\label{iv.2}
	|s_{m_{1}}-s_{m_{2}}|\leq \|s_{m_{1}}u_{n(m_{1})}-s_{m_{2}}u_{n(m_{2})}\|<\varepsilon \quad \text{for all} \;\; m_{1},m_{2}\geq m_0.
\end{equation}
Thus, there exists $s_0\in [0,1]$ such that
\begin{equation}
\label{iv.3}
  \lim_{m\to +\infty}s_m =s_0.
\end{equation}
Suppose $s_0=0$. Then, since $\|s_{m}u_{n(m)}\|=|s_m|\to 0$ as $m\to+\infty$, we find that
$$
	\lim_{m\to+\infty}\left( s_{m}u_{n(m)}\right) =0\in\mathscr{C}_{2}.
$$
Suppose $s_{0}>0$. Then, by \eqref{iv.1},
\begin{align*}
	s_0 & \|u_{n(m_1)}-u_{n(m_2)}\| = \|s_0 u_{n(m_1)}-s_0 u_{n(m_2)}\|\\ & \leq
    \|s_0 u_{n(m_1)}-s_{m_{1}}u_{n(m_1)}\|+\|s_{m_{1}}u_{n(m_{1})}
   -s_{m_{2}}u_{n(m_{2})}\|+\|s_{m_{2}}u_{n(m_{2})} -s_{0}u_{n(m_{2})}\|\\ & =
   |s_0 -s_{m_{1}}|+\|s_{m_{1}}u_{n(m_{1})}    -s_{m_{2}}u_{n(m_{2})}\|+|s_{m_{2}}-s_{0}|.
\end{align*}
Thus, by \eqref{iv.2}, we find that, whenever $m_1, m_2\geq m_0$,
$$
  s_0  \|u_{n(m_1)}-u_{n(m_2)}\| \leq |s_0 -s_{m_{1}}|+\e+|s_{m_{2}}-s_{0}|.
$$
As $\e>0$ can be taken arbitrarily small and $s_0>0$, it follows from \eqref{iv.3} that, enlarging $m_0$, if necessary,
$$
	\|u_{n(m_{1})}-u_{n(m_{2})}\|<\tfrac{1}{2} \quad \text{for all} \;\; m_{1},m_{2}\geq m_0.
$$
Thus, by \eqref{iv.1}, we can infer that $n(m_{1})=n(m_{2})$ for all $m_{1}, m_{2}\geq m_0$. Consequently,
$u_{n(m_1)}=u_{n(m_0)}$ for all $m_1 \geq m_0$. Therefore,
$$
	\lim_{m\to+\infty}\left( s_{m}u_{n(m)}\right) =s_{0}u_{n(m_0)}\in \mathscr{C}_{2},
$$
because
$$
  s_0=\lim_{m\to+\infty}s_m \leq \lim_{m\to +\infty}t_{n(m)} = 1-\tfrac{1}{n(m_0)+1}.
$$
This ends the proof that $\mathscr{C}_{2}$ is closed.
\par
Summarizing, $\mathscr{C}_{1}$ and $\mathscr{C}_{2}$ are two disjoint, closed, bounded, and connected subsets of $U$. Setting $\mathscr{B}:=\mathscr{C}_{1}\uplus \mathscr{C}_{2}$ and $\delta^{\ast}:=1$, note that $\{u_{n}\}_{n\in\N}\subset \mathscr{C}_{1}$, $\left\{t_{n} u_{n}\right\}_{n\in\N}\subset \mathscr{C}_{2}$ and
$$
	\lim_{n\to+\infty}\|u_{n}-t_{n} u_{n}\|=\lim_{n\to+\infty}\frac{1}{n+1}=0.
$$
Thus, $\text{dist}(\mathscr{C}_{1},\mathscr{C}_{2})=0$. Take $0<\d<\d^{\ast}$ and let $\mc{V}\subset U$ be any  open subset satisfying
$$
	\mathscr{C}_{1}\subset\mc{V}\subset \left\{u\in U : \text{dist}(u,\mathscr{C}_{1})<\d\right\}.
$$
We claim that  $\mathscr{B}\cap \partial\mc{V}\neq \emptyset$. Consequently, $\mathscr{C}_1$
cannot admit any open isolating neighborhood of size $\d$ with respect to $\mathscr{B}$. On the contrary, assume that
$$
  \mathscr{B}\cap \partial\mc{V}= \emptyset.
$$
In particular, $\mathscr{C}_{2}\cap \partial\mc{V}=\emptyset$. Since $\text{dist}(\mathscr{C}_{1},\mathscr{C}_{2})=0$, necessarily
\begin{equation}
\label{iv.4}
   A:=\mathscr{C}_{2}\cap\mc{V}\neq \emptyset.
\end{equation}
Moreover, since $\d\in (0,1)$, we find that $0\notin \mc{V}$. Thus,
$$
   0\in B:=\mathscr{C}_{2}\cap (U\setminus \mc{V}).
$$
On the other hand, due to \eqref{iv.4}, we have that
\begin{equation}
\label{iv.5}
	\bar{A}\subset \mathscr{C}_{2}\cap \bar{\mc{V}}=\mathscr{C}_{2}\cap \mc{V}, \quad \bar{B}\subset \mathscr{C}_{2}\cap \overline{U\setminus \mc{V}}=\mathscr{C}_{2}\cap(U\setminus \mc{V}).
\end{equation}
Therefore, since $\mathscr{C}_{2}$ is closed, \eqref{iv.5} implies that
$$
	\mathscr{C}_{2}=\bar{A}\uplus \bar{B}, \quad \bar{A}\cap \bar{B}=\emptyset.
$$
As this contradicts the connectedness of $\mathscr{C}_{2}$, it becomes apparent that $\mathscr{B}\cap \partial\mc{V}\neq \emptyset$. This argument can be easily adapted to show that $\mathscr{C}_1$ cannot admit either any open isolating neighborhood of size $\d$ with respect to $\mathscr{B}$. The proof is complete.
\end{proof}

\section{Analytic continuation theorems}

\noindent In this section we establish some continuation theorems for analytic Fredholm operators of different nature than those found in Section 3. In the analytic setting, the zero set inherits a much richer structure.
\par
Throughout this section, for any given real Banach spaces, $U$ and $V$, and any nonempty open subset $\mathcal{U}\subset \mathbb{R}\times U$, we consider an analytic map
\[
\mathfrak{F}:\mathcal{U}\longrightarrow V
\]
satisfying the following assumptions:
\begin{enumerate}
	\item[(F1)] $D_{u}\mathfrak{F}(\lambda,u)\in\Phi_{0}(U,V)$ for all $(\lambda,u)\in \mathcal{U}$;
	\item[(F2)] $\mathfrak{F}$ is proper   on closed  (in the topology of $\R\times U$) and bounded subsets of $\mathcal{U}$.
\end{enumerate}
Then,  $\mathscr{S}:=\mathfrak{F}^{-1}(0)$ is closed in $\mathcal{U}$. Through this section,  for every  $\lambda_{0}\in\mathbb{R}$, we denote
\begin{align*}
\mathcal{U}_{\lambda_{0},o}^{+} & := \mathcal{U}\cap \big[(\lambda_{0},+\infty)\times U\big],\\
\mathcal{U}_{\lambda_{0},o}^{-} & := \mathcal{U}\cap \big[(-\infty,\lambda_{0})\times U\big],\\
\mathscr{S}^{+}_{\lambda_{0},o} & :=\mathscr{S}\cap \mathcal{U}^{+}_{\lambda_{0},o}
=\{(\lambda,u)\in\mathscr{S}:\lambda>\lambda_{0}\},\\
\mathscr{S}^{-}_{\lambda_{0},o} & :=\mathscr{S}\cap \mathcal{U}^{-}_{\lambda_{0},o}
=\{(\lambda,u)\in\mathscr{S}:\lambda<\lambda_{0}\},
\end{align*}
where the subindex \lq\lq $o$\rq\rq\!  makes reference to the fact that the intervals $(\l_0,+\infty)$ and
$(-\infty,\l_0)$ are open. Obviously, $\mathscr{S}^{\pm}_{\lambda_{0},o}$ is closed in $\mathcal{U}^{\pm}_{\lambda_{0},o}$ and, by (F2), the sets $\mathscr{S}^{\pm}_{\lambda_{0},o}$ are locally compact.
\par
The next result provides us with an analytic counterpart of the first part of Theorem \ref{th3.2}. It is the main result of this section.

\begin{theorem}
\label{th5.1}
	Let $\mf{F}:\mc{U}\to V$ be an analytic operator satisfying {\rm (F1)} and {\rm (F2)}, and suppose that $(\l_{0},u_{0})\in \mf{F}^{-1}(0)$ satisfies $D_{u}\mf{F}(\l_{0},u_{0})\in GL(U,V)$. Then, there exist $\o^\pm\in \N\cup\{+\infty\}$ and two locally injective continuous curves, $ \G^{\pm}: (0,\o^\pm)\longrightarrow \mc{U}_{\l_{0},o}^{\pm}$, such that
$$
   \Gamma^{\pm}\left((0,\o^\pm)\right)\subset \mathscr{S}_{\l_{0},o}^{\pm}, \qquad \lim_{t\da 0}\Gamma^{\pm}(t)=(\l_{0},u_{0}),
$$
for which one of the following non-excluding alternatives holds:
\begin{enumerate}
		\item[{\rm (a)}]   The curve $\G^\pm$ blows up at $\omega^\pm$, in the sense that
$$
    \limsup_{t\ua \o^\pm} \|\Gamma^{\pm}(t)\|_{\mathbb{R}\times U}= +\infty.
$$
\item[{\rm (b)}]   The curve $\G^\pm$ approximates $\p \mc{U}$ as $t\to \o^\pm$, in the sense that
there exists a sequence $\{t_{n}\}_{n\in\N}$ in $(0,\o^\pm)$ such that
$$
   \lim_{n\to +\infty}t_{n} = \o^\pm \;\;\hbox{and}\;\;  \lim_{n\to+\infty} \Gamma^{\pm}(t_{n})= (\l_{\ast},u_{\ast})\in \partial\mc{U}^\pm_{\l_0,o}.
$$
		\item[{\rm (c)}]   The curve $\G^\pm$ turn backwards to the level $\l=\l_0$, in the sense that there exist $u_{1}\in \mc{U}_{\l_{0}}\setminus\{u_{0}\}$ and a sequence $\{t_{n}\}_{n\in\N}$ in $(0,\o^\pm)$ such that
$$
   \lim_{n\to +\infty} t_{n} = \o^\pm\;\; \hbox{and}\;\; \lim_{n\to+\infty}\Gamma^{\pm}(t_{n})=(\l_{0},u_{1})\in\mf{F}^{-1}(0).
$$
	\end{enumerate}
\end{theorem}

\begin{proof}
It is reminiscent of the proof of Theorem 9.1.1 of
Buffoni and Toland \cite{BT}. Attention will be focused into the construction of
$\G^+$, as the construction of $\G^-$ is analogous.
\par
A point $(\l,u)\in \mc{U}^{+}_{\l_{0},o}$ is said to be a \emph{regular zero} of $\mf{F}$ if
$$
    (\l,u)\in \mathscr{R}^{+}(\mf{F}):=D_{u}\mf{F}^{-1}(GL(U,V))\cap \mathscr{S}^{+}_{\l_{0},o}.
$$
Since $GL(U,V)$ is open and $D_{u}\mf{F}$ continuous, $\mathscr{R}^{+}(\mf{F})$ is an open subset of $\mathscr{S}^{+}_{\l_{0},o}$. By the implicit function theorem, for every $(\l,u)\in \mathscr{R}^{+}(\mf{F})$,
there is a path connected component, $\mathscr{C}_{(\l,u)}$, of $\mathscr{R}^{+}(\mf{F})$  through
the point $(\l,u)$. In this proof, any connected component $\mathscr{C}$ of $\mathscr{R}^{+}(\mf{F})$ is called a \textit{distinguished arc}. Thanks to the analytic implicit function theorem, every distinguished arc $\mathscr{C}$ is the graph of an analytic function of $\l$, i.e. there exits an open interval $I\subset\R$ and an analytic function $g: I\to U$ such that
$$
    \mathscr{C}=\{(\l,g(\l)): \l\in I \}.
$$
Based on these features, it follows from $D_{u}\mf{F}(\l_{0},u_{0})\in GL(U,V)$ that, for some
$\tilde \l_0\in (\l_0,+\infty]$, there is a maximal analytic curve $\g_0^+:(\l_0,\tilde\l_0)\longrightarrow U$ such that
$$
    \lim_{\l \da \l_{0}}\gamma_0^+(\l)=u_{0} \quad \text{and} \;\; \gamma((\l_0,\tilde\l_0))\subset \mathscr{R}^{+}(\mf{F}).
$$
By maximal, we mean that it is not strictly extensible to the right.
Subsequently, we denote by $\Gamma^{+}_0:(0,1)\longrightarrow \mathscr{S}^{+}_{\l_{0},o}$ any analytic re-parametrization of
$$
  \mathscr{C}_0:= \{(\l,\g_0^+(\l))\;:\; \l\in (\l_0,\tilde\l_0)\},
$$
to the interval $(0,1)$ preserving orientation, i.e. such that
$$
  \mathscr{C}_0:= \{\G_0^+(t) \;:\; t\in (0,1)\} \;\;\hbox{with}\;\; \lim_{t\da 0}\G_0^+(t)=(\l_0,u_0).
$$
Naturally, if $\G_0^+$ satisfies some of the alternatives (a), (b), or (c), of the statement of the theorem, we are done. Thus, suppose that $\G_0^+$ does not satisfy any of these alternatives. Then, $\mathscr{C}_0$ is bounded and separated away from $\p \mc{U}^+_{\l_0,o}$ as $t\ua 1$. Consequently, $\overline{\mathscr{C}}_{0}$ is closed, bounded and it is contained in $\mc{U}$. Therefore, by (F2), there exist
$(\l_1,u_1)\in \mc{U}$ and a sequence $\{t_n\}_{n\geq 1}$ in $(0,1)$ such that
\begin{equation}
\label{5.1}
  \lim_{n\to +\infty}\G^+_0(t_n)=(\l_1,u_1).
\end{equation}
Note that $\l_{1}\neq \l_{0}$ by the uniqueness of the implicit function theorem applied to the point $(\l_0,u_0)$ and the exclusion of alternative (c). Therefore $(\l_1,u_1)\in\mc{U}^{+}_{\l_0,o}$. By this fact and the maximality of $\g_0^+$, $(\l_1,u_1)$ must be a singular zero of $\mathfrak{F}$. Thus,
\begin{equation}
\label{5.2}
   (\l_1,u_1)\in \mathscr{S}_{\l_0,o}^+\backslash \mathscr{R}^{+}(\mf{F}).
\end{equation}
Next, we will analyze  the structure of $\mathscr{S}^{+}_{\l_{0},o}$ in a neighborhood of $(\l_{1},u_{1})$. By construction, $\l_1>\l_0$ and, since $D_u\mf{F}(\l_1,u_1)\in\Phi_0(U,V)$, we have that
\begin{equation}
\label{5.3}
   1\leq n:=\mathrm{dim\,}N[D_{u}\mathfrak{F}(\lambda_{1},u_{1})]=
   \mathrm{codim\,}R[D_{u}\mathfrak{F}(\lambda_{1},u_{1})].
\end{equation}
Thus, there are two linear continuous projections
\begin{equation*}
		P:U \longrightarrow N[D_{u}\mathfrak{F}(\lambda_{1},u_{1})], \qquad Q:V \longrightarrow R[D_{u}\mathfrak{F}(\lambda_{1},u_{1})],
\end{equation*}
and, once fixed these projections, we can decompose
\begin{equation}
\label{5.4}
\begin{split}
	& U= N[D_{u}\mathfrak{F}(\lambda_{1},u_{1})] \oplus Y, \quad Y=N[P],\\
	& V =Z \oplus R[D_{u}\mathfrak{F}(\lambda_{1},u_{1})],  \quad  \ Z=N[Q].
\end{split}
\end{equation}
Thanks to \eqref{5.3},  $\mathrm{dim\,}Z=n$, and, hence, we can identify $\mathbb{R}\times N[D_{u}\mathfrak{F}(\lambda_{1},u_{1})]$ with $\mathbb{R}^{n+1}$ via a (fixed) linear isomorphism
\begin{equation}
\label{5.5}
	T:\mathbb{R}\times N[D_{u}\mathfrak{F}(\lambda_{1},u_{1})] \longrightarrow \R\times \R^{n}, \quad (\l,x)\mapsto T(\l,x):=(\l, Lx).
\end{equation}
For instance, once chosen  a basis in $N[D_{u}\mathfrak{F}(\lambda_{1},u_{1})]$, $Lx$ might be the coordinates of $x \in N[D_{u}\mathfrak{F}(\lambda_{1},u_{1})]$ with respect to that basis.
Similarly, we can identify $Z$ with $\R^{n}$ via another (fixed) linear isomorphism
$$
   S: Z \longrightarrow \R^{n}.
$$
By construction, any element $u\in U$ admits a unique decomposition as
\begin{equation}
\label{5.6}
	u = u_{1}+x + y, \qquad x = P (u-u_{1}), \quad y = (I_U-P)(u-u_{1}),
\end{equation}
and the equation $\mf{F}(\l,u)=0$ can be equivalently expressed as
\begin{equation}
\label{5.7}
	\left\{ \begin{array}{l} Q \mathfrak{F}(\l,u_{1}+x+y)=
		0,  \\[4pt]	(I_V-Q)\mathfrak{F}(\l,u_{1}+x+y)=  0. \end{array}\right.
\end{equation}
Now, adopting the methodology of Section 3.1 of \cite{LG01}, we consider the analytic operator
\begin{equation*}
	\mc{H}\, :\, \R \times N[D_{u}\mathfrak{F}(\lambda_{1},u_{1})]\times Y \to V, \quad \mc{H}(\l,x,y):=Q \mathfrak{F}(\l,u_{1}+x+y).
\end{equation*}
Since $\mc{H}(\l_1,0,0)=0$ and the linearization
$$
D_y\mc{H}(\l_1,0,0) = Q D_{u}\mathfrak{F}(\lambda_{1},u_{1})|_Y\, :\,
Y \longrightarrow R[D_{u}\mathfrak{F}(\lambda_{1},u_{1})]
$$
is an isomorphism, by the implicit function theorem, there exist a neighborhood $\mc{V}$ of $(\l_1,0)$ in $\R\times N[D_{u}\mathfrak{F}(\lambda_{1},u_{1})]$ and an analytic map $\psi:\mc{V}\to Y$ such that
\begin{equation}
\label{5.8}
	\mc{H}(\l,x,\psi(\l,x))=0\;\; \hbox{for all}\;\; (\l,x)\in\mc{V}.
\end{equation}
Moreover, there is a neighborhood
$\mc{W}$ of $(\l,u)=(\l_1,u_{1})$ in $\mc{U}_{\l_0,o}^{+}$ such that
\begin{equation}
\label{5.9}
  y = \psi(\l,x) \;\;\hbox{if}\;\; (\l,u)=(\l,u_{1}+x+y)\in \mc{W}\;\;\hbox{and}\;\; \mc{H}(\l,x,y)=0.
\end{equation}
Thus, since $\mc{H}(\l_1,0,0)=0$, we find that $\psi(\l_{1},0)=0$. Finally, substituting $y=\psi(\l,x)$ into the second equation of \eqref{5.7} yields
\begin{equation}
\label{5.10}
	(I_V - Q)\mathfrak{F}(\l,u_{1}+x+\psi(\l,x))=0, \quad (\l,x)\in \mc{V}.
\end{equation}
Therefore, $(\l,x)\in\mc{V}$ solves \eqref{5.10} if and only if
$$
(\l,u)=(\l,u_{1}+x+\psi(\l,x))\in\mc{W}
$$
satisfies  $\mf{F}(\l,u)=0$. Consequently,  introducing the open neighborhood of $(\l_{1},0)$
$$
   \mc{E}:=T(\mc{V})=\{(\l,Lx) : (\l,x)\in\mc{V}\}\subset \R\times \R^{n}
$$
as well as the analytic map $\mathfrak{G}: \mc{E}\subset\mathbb{R}\times \R^{n}\longrightarrow \mathbb{R}^{n}$
defined by
\begin{equation}
\label{5.11}
	\mf{G}(\lambda,z):=S(I_{V}-Q)\mathfrak{F}(\lambda,u_{1}+L^{-1}z+\psi(\lambda,L^{-1}z)),
\end{equation}
it becomes apparent that solving $\mf{F}(\l,u)=0$ in $\mc{W}$ is equivalent to
solve the finite-dimensional equation $\mf{G}(\l,z)=0$ in $\mc{E}$. Note that $\mf{G}(\l_{1},0)=0$.
Moreover, by construction, the maps
\begin{align*}
		& \Psi: \; \mathfrak{F}^{-1}(0)\cap \mc{W}\longrightarrow \mathfrak{G}^{-1}(0),
        \hspace{0.8cm} (\lambda,u)\mapsto (\lambda,LP(u-u_{1})),\\
		 & \Psi^{-1}:\; \mathfrak{G}^{-1}(0)\longrightarrow \mathfrak{F}^{-1}(0)\cap\mc{W},
      \quad (\lambda,z)\mapsto (\lambda,u_{1}+L^{-1}z+\psi(\lambda,L^{-1}z)),
\end{align*}
are inverses of each other. Furthermore, for every $(\l,z)\in \mf{G}^{-1}(0)$,
\begin{equation*}
    D_{z}\mathfrak{G}(\l,z)\in GL(\R^{n})\;\; \hbox{if and only if} \;\;
    D_{u}\mathfrak{F}(\lambda,u_1+ L^{-1}z+\psi(\lambda,L^{-1}z))\in GL(U,V).
\end{equation*}
Consequently,
$$
   \mathscr{R}(\mf{G}):= \Psi(\mathscr{R}^{+}(\mf{F})\cap \mc{W})
$$
provides us with the set of regular points of $\mf{G}$, i.e.
$$
   \mathscr{R}(\mf{G})=\{(\l,z)\in\mf{G}^{-1}(0) : \;D_{z}\mf{G}(\l,z)\in GL(\R^n)\}.
$$
Subsequently, we set	
$$
    W:=\mathscr{V}(\mc{E},\{\mf{G}\})=\mf{G}^{-1}(0), \quad M:=\mathscr{R}(\mf{G}).
$$
Clearly, $M\subset W$ is a one-dimensional real analytic manifold. Let
$$
   \mc{M}:=\{M_{j} : j\in J\}
$$
be the family of connected components of $M$ such that the germ of $M_j$ at $(\l_1,0)$, $\gamma_{(\l_{1},0)}(M_{j})$, is nonempty. The real analytic function $\mf{G}(\l,z)$ can be \textit{complexified}  by replacing $(\l,z)\in \R^{n+1}$ with $(\l,z)\in\C^{n+1}$, which leads
to a real-on-real analytic function
$$
   \mf{G}^{\mf{c}}:\mc{E}^{\mf{c}}\subset \C\times \C^{n} \longrightarrow \C^{n}
$$
defined in a sufficiently small complex neighborhood $\mc{E}^{\mf{c}}$ of $(\l_1,0)$ with  $\mc{E}\subset\mc{E}^{\mf{c}}$. Lastly, we introduce the complex counterparts of $W$ and $M$,
$$
  W^{\mf{c}}:=\mathscr{V}(\mc{E}^{\mf{c}},\{\mf{G}^{\mf{c}}\}),\quad
    M^{\mf{c}}:=\mathscr{R}(\mf{G}^{\mf{c}}) \equiv
    \{(\l,z)\in \mc{E}^{\mf{c}} : D_{z}\mf{G}^{\mf{c}}(\l,z)\in GL(\C^n)\},
$$
where $U^\mf{c}:=U+iU$ and $V^\mf{c}:=V+iV$. Clearly, $M^{\mf{c}}$ is a one-dimensional complex analytic manifold. Let us denote by
$$
    \mc{M}^{\mf{c}}:=\{M_{j}^{\mf{c}} : j\in J^{\mf{c}}\},
$$
the connected components of $M^{\mf{c}}$ with nonempty germ $\gamma_{(\l_{1},0)}(\R^{n+1}\cap M_{j}^{\mf{c}})$. Naturally, for every $j\in J$, there exists $j_{0}=j_0(j)\in J^{\mf{c}}$ such that $M_{j}\subset M_{j_{0}}^{\mf{c}}$.
\par
According to Theorem \ref{thA.1} (iv)--(vi) applied to $W^{\mf{c}}$, it becomes apparent that, for every $j\in J^{\mf{c}}$, there is a real-on-real branch $B_{j}$ of a Weierstrass analytic variety such that
$$
    \gamma_{(\l_{1},0)}(M^{\mf{c}}_{j})\subset \gamma_{(\l_{1},0)}(\overline{B}_{j}), \quad \dim B_{j}=1, \quad B_{j}\subset W^{\mf{c}}.
$$
Moreover, by shortening the neighborhood $\mc{E}^{\mf{c}}$, if necessary, we can suppose that
$$
    B_{j}\backslash \{(\l_{1},0)\}\subset M^{\mf{c}}_{j}.
$$
By Theorem \ref{thA.1},  there are finitely many branches and, hence, $\mc{M}$ and $\mc{M}^{\mf{c}}$ have finitely many components. By Theorem \ref{thA.2}, each of these one dimensional branches $B_{j}$ admits an injective continuous complex parametrization in a neighborhood of $(\l_{1},0)$ as a Puiseux series. Moreover, for every $j\in J^{\mf{c}}$, 	$\R^{n+1}\cap \overline{B}_{j}$ admits an injective continuous real parametrization of the form \eqref{A.2}. Consequently, in a neighborhood of $(\l_{1},0)\in \mc{E}$, $\overline{M}$ consists of the graphs of finitely many curves passing through $(\l_{1},0)\in \mc{E}$ that intersect to each other only at $(\l_{1},0)$ and are given by some parametrization of the type \eqref{A.2}. Therefore, each $M_{j}$, $j\in J$, is paired in an unique way with another $M_{\sigma(j)}$, $\sigma(j)\in J\backslash\{j\}$, so that their union with the point $(\l_{1},0)$ forms one of these curves. Based on these features, it is easily seen that  \eqref{5.1} and \eqref{5.2} imply
\begin{equation}
\label{5.12}
\lim_{t\ua 1}\G^+_0(t)=(\l_1,u_1)\in \mathscr{S}_{\l_0,o}^+\backslash \mathscr{R}^{+}(\mf{F}).
\end{equation}
Moreover, there exist $\d>0$ and $M_{j}\in\mc{M}$ such that
$$
   \Psi(\{\G_0^+(t)\;:\; t\in (1-\d,1)\})= M_{j}.
$$
Thus, there exists a unique ${M}_{\sigma(j)}$, $\sigma(j)\in J\backslash\{j\}$, for which
$$
    M_{j}\cup \{(\l_{1},0)\}\cup M_{\sigma(j)}
$$
 is the graph of an injective continuous curve
$$
    \Sigma:(1-\d,1+\d) \longrightarrow M_{j}\cup \{(\l_{1},0)\}\cup M_{\sigma(j)},
$$
$$
 \Psi^{-1}\circ\Sigma|_{(1-\d,1)}=\G^{+}_{0}|_{(1-\d,1)}, \quad \Sigma(1)=(\l_1,0), \quad
\Sigma((1,1+\d))\subset M_{\s(j)},
$$
of the form \eqref{A.2}. Therefore, the injective continuous curve
$$
     \Theta:(1-\d,1+\d)\longrightarrow \mf{F}^{-1}(0), \quad \Theta:=\Psi^{-1}\circ\Sigma,
$$
extends $\G_0^+$ uniquely beyond $(\l_{1},u_{1})$. Moreover, we can suppose that
$$
   \Theta((1,1+\d))\subset \mathscr{R}^{+}(\mf{F}),
$$
shortening $\d$, if necessary.
\par
Let $\G_1^+:(0,2)\to \mathscr{S}^{+}_{\l_0,o}$ be a maximal injective and continuous curve
such that
\begin{equation}
\label{5.13}
\begin{split}
   \G_1^+(t) & =\G_0^+(t) \;\; \hbox{for all}\;\; t \in (0,1), \\
  \G_1^+(1) & = (\l_1,u_1),\\
  \G_1^+(t) & \in  \mathscr{R}^{+}(\mf{F}) \;\; \hbox{for all}\;\; t \in (1,2).
\end{split}
\end{equation}
Set
$$
   \mathscr{C}_{1}:= \{\G_1^+(t) : t\in (1,2)\}.
$$
Then,  $\mathscr{C}_{0}\cap \mathscr{C}_{1}=\emptyset$. Indeed, if $\mathscr{C}_{0}\cap \mathscr{C}_{1}\neq\emptyset$, by the uniqueness of the implicit function theorem, necessarily $\mathscr{C}_{0}=\mathscr{C}_{1}$. But this cannot happen, since $M_{j}\neq M_{\s(j)}$.
If $\G_1^+$ satisfies some of the alternatives (a), (b), or (c) of the theorem, we are done. If not, repeating the previous argument, we find that
$$
  \lim_{t\ua 2} \G_1^+(t)=(\l_2,u_2)\in \mathscr{S}_{\l_0,o}^+\backslash \mathscr{R}^{+}(\mf{F})
$$
and the previous argument can be repeated up to construct  a maximal locally injective and continuous curve,
$\G_2^+:(0,3)\to \mathscr{S}^{+}_{\l_{0},o}$, such that
\begin{equation}
\label{5.13}
\begin{split}
   \G_2^+(t) & =\G_1^+(t) \;\; \hbox{for all}\;\; t \in [0,2), \\
  \G_2^+(2) & = (\l_2,u_2),\\
  \G_2^+(t) & \in \mathscr{R}^{+}(\mf{F})\;\; \hbox{for all}\;\; t \in (2,3).
\end{split}
\end{equation}
Although $\G_2^+$ is locally injective,  because it is injective on $(0,2)\cup (2,3)$, it might be not injective, because it might happen that $(\l_0,u_0)=(\l_1,u_1)$, i.e. $\Gamma^{+}_{2}(1)=\Gamma^{+}_{2}(2)$.
\par
Setting
$$
   \mathscr{C}_{2}:= \{\G_2^+(t) : t\in (2,3)\},
$$
we claim that
\begin{equation}
\label{jlg1}
   \mathscr{C}_{0}\cap \mathscr{C}_{2}=\emptyset\;\;
   \hbox{and}\;\; \mathscr{C}_{1}\cap \mathscr{C}_{2}=\emptyset.
\end{equation}
Arguing by contradiction, assume that $\mathscr{C}_{1} \cap \mathscr{C}_{2} \neq \emptyset$. Then, by the uniqueness provided by the implicit function theorem, we must have $\mathscr{C}_{1} = \mathscr{C}_{2}$.
Thus, by the unique continuation property applied at the singular point $(\lambda_{2}, u_{2})$, it follows that
$(\lambda_{1}, u_{1}) = (\lambda_{2}, u_{2})$. Consequently, $\mathscr{C}_{1} = \mathscr{C}_{2}$ is a closed loop whose only singular point is $(\lambda_{1}, u_{1})$. However, adapting  the same argument, this would also imply that $\mathscr{C}_{1} = \mathscr{C}_{0}$, which is impossible. Therefore, $\mathscr{C}_{1} \cap \mathscr{C}_{2} = \emptyset$, as claimed above. Similarly, if $\mathscr{C}_{0}\cap \mathscr{C}_{2}\neq \emptyset$, then  $\mathscr{C}_{0} = \mathscr{C}_{2}$ and, necessarily, $(\lambda_{1}, u_{1}) = (\lambda_{2}, u_{2})$.
Indeed, if $(\lambda_{1}, u_{1}) \neq (\lambda_{2}, u_{2})$, then we would have
$(\lambda_{2}, u_{2}) \in \overline{\mathscr{C}}_{2} \setminus \overline{\mathscr{C}}_{0}$, a contradiction, because $\mathscr{C}_{1}$ would again be a closed loop with $(\lambda_{1}, u_{1})$ as its unique singular point, and,  by the same reasoning as above, we would obtain that $\mathscr{C}_{1} = \mathscr{C}_{0}$, which is impossible. Therefore, \eqref{jlg1} holds.
\par
Arguing inductively, the proof is complete, unless the previous process does not provide us
with a curve satisfying some of the alternatives of the theorem. In such case, for every integer $n\geq 2$, there exists a maximal locally injective and continuous curve $\G_n^+:(0,n+1)\to \mathscr{S}^{+}_{\l_{0},o}$ such that
\begin{equation}
\label{5.13}
\begin{split}
   \G_n^+(t) & =\G_{n-1}^+(t) \;\; \hbox{for all}\;\; t \in (0,n), \\
  \G_n^+(n) & = (\l_{n},u_{n})\in \mathscr{S}_{\l_0,o}^+\backslash \mathscr{R}^{+}(\mf{F}),\\
  \G_{n}^+(t) & \in \mathscr{R}^{+}(\mf{F})\;\; \hbox{for all}\;\; t \in (n,n+1),
\end{split}
\end{equation}
$$
  \lim_{t\ua n+1} \G_n^+(t)=(\l_{n+1},u_{n+1})\in \mathscr{S}_{\l_0,o}^+\backslash \mathscr{R}^{+}(\mf{F}),
$$
and the arcs $\{\mathscr{C}_{j}\}_{j = 0}^{n}$ defined by
$$
   \mathscr{C}_{j}:= \{\G_{j}^+(t) : t\in (j,j+1)\},\quad 0\leq j\leq n,
$$
are pairwise disjoint. Suppose this occurs for all integer $n\geq 2$, and consider the limiting curve
$$
 \G_\o^+:(0,+\infty)\to \mathscr{S}^{+}_{\l_0,o}
$$
defined by
$$
	\G_\o^{+}(t):= \left\{ \begin{array}{ll} \G^{+}_{1}(t) & \hbox{if} \;\; t\in (0,2),\\[1ex]
     \G^{+}_{n}(t) & \hbox{if} \;\; t\in [n,n+1) \;\;\hbox{with}\;\;n\geq 2.\end{array}\right.
$$
If $\G_\o^+$ satisfies some of the alternatives of the theorem, we are done. If not,  since the arcs $\{\mathscr{C}_{n}\}_{n\in\N}$ are pairwise disjoint, the sequence $\{\G_\o^+(\tfrac{n}{2})\}_{n\geq 1}$ consists of distinct points. As we are assuming that $\G_\o^+$ does not satisfy any of the
alternatives of the theorem, $\{\G_\o^+(\tfrac{n}{2})\}_{n\geq 1}$  is bounded and separated away from $\p \mc{U}_{\l_0,o}^+$. Thus,
there is a subsequence, $\{\G_\o^+(\tfrac{n_m}{2})\}_{m\geq 1}$, such that
$$
\lim_{m\to +\infty}\G_\o^+(\tfrac{n_m}{2})=(\l_\o,u_\o)\in \mathscr{S}_{\l_0,o}^+.
$$
Since $\G_\o^+(\tfrac{n_m}{2})\in \mathscr{R}^{+}(\mf{F})$ for all $m\geq 1$, this implies that, in a neighborhood of $(\l_\o,u_\o)$, there are infinitely many distinct analytic arcs, which contradicts Theorem \ref{thA.1} after a Lyapunov--Schmidt reduction of $\mf{F}$ on the point $(\l_\o,u_\o)$. This ends the proof.
\end{proof}

\section{An application}

\noindent Let $\Omega \subset \mathbb{R}^{N}$ be a bounded open set with boundary, $\p\O$, of class $\mc{C}^3$, and suppose that $q\in\R$, $q>1$, and $\mu>\sigma_{1}$, where $\sigma_{1}$ stands for the principal eigenvalue of
$$
  -\D \equiv -\sum_{j=1}^N \frac{\p^2}{\p x_j^2}\quad \hbox{in}\;\;\O
$$
under homogeneous Dirichlet boundary conditions on $\p\O$. In this section, we analyze the quasilinear
parametric boundary value problem
\begin{equation}
	\label{vi.1}
	\left\{
	\begin{array}{ll}
		-\operatorname{div}\!\left(\dfrac{\nabla u}{\sqrt{1-\lambda |\nabla u|^{2}}}\right)
		= \mu u - u^{q} & \quad \text{in } \Omega, \\[1.2ex]
		u=0 & \quad \text{on } \partial\Omega,
	\end{array}
	\right.
\end{equation}
where $\lambda \in \mathbb{R}$ is regarded as a bifurcation parameter. The left-hand side of \eqref{vi.1} corresponds to the mean curvature operator when $\lambda<0$,
to the Laplacian $-\Delta$ when $\lambda=0$, and to the Minkowski  operator when $\lambda>0$.
\par
To the best of our knowledge, the existing literature treats separately the prescribed (Euclidean) mean curvature equation and its Minkowski counterpart, but does not address the mixed problem considered here. In this work we deal with both cases within a unified framework, constructing a global continuum of positive solutions that connects the Euclidean mean curvature regime with the Minkowski regime. For a comprehensive account of the prescribed (Euclidean) mean curvature equation we refer to \cite{MC1,MC2,MC3,MC4,MC5,MC6,MC7,MC8,MC9,Se}, whereas for the Minkowski mean curvature equation we refer to \cite{MM1,MM2,MM3,MM4,MM5,MM6,MM7,MM8,MM9,MM10,MM11,MM12,MM13,MM15,MM16}.
\par
Problem \eqref{vi.1} in the case $\lambda = -1$ has been studied in \cite{OM,Se2} and the references therein. In the present paper we consider this problem as a model application of our abstract result. A more detailed analysis of this equation will be carried out in a forthcoming work.
\par
Throughout this section, the solutions of \eqref{vi.1} are regarded as zeroes of the nonlinear operator
$\mathfrak{F}:\mathcal{U}_{\delta} \longrightarrow L^{p}(\Omega)$ defined by
$$
\mathfrak{F}(\lambda,u):=-\operatorname{div}\!\left(\dfrac{\nabla u}{\sqrt{1-\lambda |\nabla u|^{2}}}\right)
- \mu u + u^{q},\quad (\l,u)\in \mathcal{U}_{\delta},
$$
where, for a given (fixed) $\d\in (0,1)$,  we are denoting
$$
\mathcal{U}_{\delta}:=\left\{(\lambda,u)\in \mathbb{R}\times W^{2,p}_{0}(\Omega)
:\; 1-\lambda \|\nabla u\|^{2}_{\infty}>\delta\right\}.
$$
Here, $W^{2,p}_{0}(\Omega):=W^{2,p}(\Omega)\cap W^{1,p}_{0}(\Omega)$, with $p>N$, are the usual
Sobolev spaces.  As the embedding $W^{2,p}_{0}(\Omega)\hookrightarrow \mathcal{C}^{1,1-\frac{N}{p}}(\bar{\Omega})$ is continuous, $\mathcal{U}_{\delta}$ is an open subset of $\mathbb{R}\times \mathcal{C}^{1,1-\frac{N}{p}}(\bar{\Omega})$. Since
$$
   1-\lambda \|\nabla u\|^{2}_{\infty}\geq 1 >\delta
$$
for all $\l\leq 0$, we have that
\begin{equation}
	\label{Eq6.2}
    (-\infty,0]\times W^{2,p}_{0}(\Omega)\subset \mathcal{U}_{\delta}.
\end{equation}
It is easily seen that $\mathfrak{F}\in \mathcal{C}^{1}(\mathcal{U}_{\delta}, L^{p}(\Omega))$,
with Fr\'{e}chet derivative given by
\[
D_{u}\mathfrak{F}(\lambda,u)[v]
= -\operatorname{div}\!\left(
\left[\frac{I_{N}}{\sqrt{1-\lambda |\nabla u|^{2}}}
+\lambda \frac{\nabla u \otimes \nabla u}{(1-\lambda |\nabla u|^{2})^{3/2}}
\right]\nabla v\right)
- \mu v+ q u^{q-1} v
\]
for every $(\lambda,u)\in\mathcal{U}_{\delta}$ and $v\in W^{2,p}_{0}(\Omega)$,
where $I_{N}$ denotes the identity matrix in $\mathbb{R}^{N}$ and
\[
\nabla u \otimes \nabla u:=\big(\partial_{x_{i}}u\,\partial_{x_{j}}u\big)_{i,j}\in \mathrm{Sym}(N,\R).
\]
Moreover, $\mathfrak{F}: \mathcal{U}_{\delta}\to L^{p}(\Omega)$ is analytic if $q\geq 2$ is an integer.

\begin{lemma}
\label{le6.1}
	For every $(\l,u)\in\mc{U}_{\d}$, $D_{u}\mf{F}(\l,u)$ is a Fredholm operator of index zero.
\end{lemma}

\begin{proof}
For every $(\l,u)\in \mc{U}_{\d}$, we can rewrite
$$
	D_{u}\mf{F}(\l,u)[v]=-\text{div}(A(x)\nabla v)+c(x)v, \quad v\in W^{2,p}_{0}(\O),
$$
where, for every $x\in\O$,
$$
A(x):=\frac{I_{N}}{\sqrt{1-\l |\nabla u(x)|^{2}}}+\l \frac{(\nabla u \otimes \nabla u)(x)}{(1-\l |\nabla u(x)|^{2})^{\frac{3}{2}}}, \quad c(x):=-\mu+q u^{q-1}(x).
$$
Note that $A\in\text{Sym}(N,\R)$ and $c\in \mc{C}(\bar{\O})$, because $q>1$ and
$u\in\mathcal{C}^{1,1-\frac{N}{p}}(\bar{\Omega})$. Moreover, since  $u\in\mathcal{C}^{1,1-\frac{N}{p}}(\bar{\Omega})$, the coefficients of the matrix $A(x)\equiv \left(a_{ij}(x)\right)_{1\leq i,j\leq N}$ are H\"{o}lder continuous. In particular, they are bounded.
Actually, since $u\in \mc{U}_\d$, by the definition of $\mc{U}_\d$, we have that
$$
	\frac{1}{1-\l |\nabla u(x)|^{2}}<\frac{1}{\d} \quad \text{for all} \;\; x\in \bar{\O}.
$$
Thus, for every $i, j\in\{1,\ldots,N\}$ and $x\in\O$, we find that
\begin{equation*}
		|a_{ij}(x)|=\left|\frac{\d_{ij}}{\sqrt{1-\l |\nabla u(x)|^{2}}}+\l \frac{\partial_{x_{i}}u(x)\cdot \partial_{x_{j}}u(x)}{(1-\l |\nabla u(x)|^{2})^{\frac{3}{2}}}\right|\leq \frac{1}{\d^{\frac{1}{2}}}+|\l| \frac{ \|\nabla u\|^{2}_{\infty}}{\d^{\frac{3}{2}}}<+\infty.
\end{equation*}
Hence, $a_{ij}\in \mc{C}^{0,\b}(\bar\O)$ for some $\b\in (0,1)$, and $D_{u}\mf{F}(\l,u)$ is a differential operator with bounded H\"{o}lder continuous coefficients. Now, we will show that $D_{u}\mf{F}(\l,u)$ is uniformly elliptic. Indeed, for every  $\xi=(\xi_{i})_{i=1}^{N}\in \R^{N}$ such that $\xi\neq 0$, one has that
\begin{align*}
	\sum_{i,j=1}^{N}a_{ij}(x)\xi_{i}\xi_{j}=\frac{|\xi|^{2}}{\sqrt{1-\l |\nabla u(x)|^{2}}}+\l \frac{
\langle \xi,\nabla u(x)\rangle^{2}}{(1-\l |\nabla u(x)|^{2})^{\frac{3}{2}}},
\end{align*}
where $\langle\cdot,\cdot\rangle$ stands for the Euclidean product of $\R^N$. Next, we will distinguish three cases according to the sign of the parameter $\l$. Suppose $\l=0$. Then, for every $x\in\O$,
$$
	\sum_{i,j=1}^{N}a_{ij}(x)\xi_{i}\xi_{j}=|\xi|^{2}.
$$
Suppose $\l>0$. Then, for every $x\in\O$,
\begin{equation*}
		\sum_{i,j=1}^{N}a_{ij}(x)\xi_{i}\xi_{j}\geq \frac{|\xi|^{2}}{\sqrt{1-\l |\nabla u(x)|^{2}}}\geq \frac{|\xi|^{2}}{\sqrt{1-\l \|\nabla u \|^{2}_{\infty}}}.
\end{equation*}
Lastly, in case  $\l<0$, by the Cauchy--Schwarz inequality, we find that, for every $x\in\O$,
\begin{align*}
    	\sum_{i,j=1}^{N}a_{ij}(x)\xi_{i}\xi_{j} &=\frac{|\xi|^{2}}{\sqrt{1-\l |\nabla u(x)|^{2}}}-|\l| \frac{\langle\xi,\nabla u(x)\rangle^{2}}{(1-\l |\nabla u(x)|^{2})^{\frac{3}{2}}} \\ &
    	\geq \frac{|\xi|^{2}}{\sqrt{1-\l |\nabla u(x)|^{2}}}-|\l| \frac{|\xi|^{2}\cdot |\nabla u(x)|^{2}}{(1-\l |\nabla u(x)|^{2})^{\frac{3}{2}}} \\
    	& = \frac{|\xi|^{2}}{(1-\l |\nabla u(x)|^{2})^{\frac{3}{2}}}\geq \frac{|\xi|^{2}}{(1-\l \|\nabla u\|^{2}_{\infty})^{\frac{3}{2}}}.
\end{align*}
Consequently, $D_{u}\mf{F}(\l,u)$ is uniformly elliptic in $\O$ with ellipticity constant
$$
\theta:=\min\left\{1, \frac{1}{\sqrt{1-\l \|\nabla u \|^{2}_{\infty}}}, \frac{1}{(1-\l \|\nabla u\|^{2}_{\infty})^{\frac{3}{2}}}\right\}>0.
$$
Thus, combining the Lax--Milgram theorem, \cite{LaxM}, with some standard techniques in $L^p$--regularity theory (see Denk, Hieber and Pr\"{u}ss \cite{DHPa,DHPb}), we find that, for sufficiently large $\nu>0$,
$$
D_{u}\mf{F}(\l,u)+\nu J \in GL(W^{2,p}_{0}(\O), L^{p}(\O)),
$$
where $J: W^{2,p}_{0}(\O)\hookrightarrow L^{p}(\O)$ is the canonical embedding, which is compact.  Therefore,
$$
D_{u}\mf{F}(\l,u) = [D_{u}\mf{F}(\l,u)+\nu J]-\nu J
$$
is a compact perturbation of an invertible operator. As these operators are Fredholm of index zero,
the proof is completed.
\end{proof}

\begin{lemma}
\label{le6.2}
$\mf{F}:\mc{U}_{\d}\to L^{p}(\O)$ is proper in closed and bounded subsets of $\mc{U}_{\d}$.
\end{lemma}

\begin{proof}
It suffices to show that the restriction of $\mf{F}$ to the closed subset
${K}:=[\lambda_{-},\lambda_{+}]\times \bar {B}_{R}$ is proper, where $\lambda_{-}<\lambda_{+}$ are arbitrary and $B_{R}$ stands for the open ball of $W^{2,p}_{0}(\O)$ of radius $R>0$ centered at $0$; the parameters $\l_{\pm}$ and $R>0$ are chosen in such a way that $K\subset \mc{U}_{\d}$. Thanks to Theorem 2.7.1 of Berger \cite{Ber}, it suffices to check that $\mf{F}({K})$ is closed in $L^{p}(\O)$ and that, for every $f\in L^{p}(\O)$,  $\mf{F}^{-1}(f)\cap {K}$ is compact in $\mathbb{R}\times W^{2,p}_{0}(\O)$.
To show that $\mf{F}(K)$ is closed in $L^{p}(\O)$, let $\{f_{n}\}_{n\in\mathbb{N}}$ be a sequence in $\mf{F}(K)\subset L^{p}(\O)$ such that
\begin{equation}
\label{vi.2}
		\lim_{n\to\infty} f_{n}=f \quad \text{in } L^{p}(\O).
\end{equation}
Then, there exists a sequence $\{(\l_{n},u_{n})\}_{n\in\mathbb{N}}$ in $K$ such that
\begin{equation}
\label{vi.3}
		f_{n}=\mf{F}(\l_{n},u_{n}) \quad \hbox{for all} \;\; n\in\mathbb{N}.
\end{equation}
According to a classical result of Rellich \cite{Re} and Kondrachov \cite{Ko}, for every $\nu<1-\frac{N}{p}$,
the canonical imbedding $W^{2,p}(\O)\hookrightarrow \mc{C}^{1,\nu}(\bar\O)$ is compact (see
Theorem 4.5 of \cite{LG13}). Thus, we can extract a subsequence $\{(\l_{n_{k}},u_{n_{k}})\}_{k\in\mathbb{N}}$ such that, for some $(\l_{\o},u_{\o})\in [\l_-,\l_+]\times \mc{C}^{1,\nu}(\bar\O)$,
\begin{equation}
\label{vi.4}
		\lim_{k\to \infty} \l_{n_{k}} = \l_{\o} \quad \hbox{and}\quad
\lim_{k\to\infty} u_{n_{k}} =u_{\o} \quad \text{in } \mc{C}^{1,\nu}(\bar\O).
\end{equation}
As a direct consequence of \eqref{vi.2}, \eqref{vi.3} and \eqref{vi.4}, it becomes apparent that
$u_{\o}$ must be a weak solution of the linear elliptic problem
\begin{equation}
\label{vi.5}
\left\{\begin{array}{ll}
			-\operatorname{div}\!\left(\dfrac{\nabla u}{\sqrt{1-\lambda_\o |\nabla u_\o|^{2}}}\right)
			-(\mu - u_{\o}^{q-1})u = f & \text{in } \Omega, \\[1.2ex]
			u=0 & \text{on } \partial\Omega.
\end{array}		\right.
\end{equation}
By elliptic regularity, $u_{\o}\in W^{2,p}_{0}(\O)$ and $f=\mf{F}(\l_{\o},u_{\o})$. Therefore, $f\in \mf{F}(K)$.
\par
Now, pick $f\in L^{p}(\O)$. To show that $\mf{F}^{-1}(f)\cap K$ is compact in $[\l_-,\l_+]\times W^{2,p}_{0}(\O)$, let $\{(\l_{n},u_{n})\}_{n\in\mathbb{N}}$ be a sequence in $\mf{F}^{-1}(f)\cap K$. Then,
\begin{equation}
\label{vi.6}
		\mf{F}(\l_{n},u_{n})=f \quad \hbox{for all}\;\; n\in\mathbb{N}.
\end{equation}
Again by the compactness of the imbedding $W^{2,p}(\O)\hookrightarrow \mc{C}^{1,\nu}(\bar\O)$,
there is a subsequence $\{(\l_{n_{k}},u_{n_{k}})\}_{k\in\mathbb{N}}$ such that \eqref{vi.4} holds for some $(\l_{\o},u_{\o})\in [\l_-,\l_+]\times \mc{C}^{1,\nu}(\bar\O)$. Thus, reasoning as above,
$u_{\o}\in \mc{C}^{1,\nu}(\bar\O)$ is a weak solution of \eqref{vi.5} and, by elliptic regularity,
$u_{\o}\in W^{2,p}_{0}(\O)$ and $\mf{F}(\l_{\o},u_{\o})=f$. In particular, for every $k\in\mathbb{N}$,
\begin{equation*}
		-\text{div}\left(\frac{\nabla u_{n_k}}{\sqrt{1-\l_{n_k}|\nabla u_{n_k}|^{2}}}-
        \frac{\nabla u_\o}{\sqrt{1-\l_\o |\nabla u_\o|^{2}}}\right)=\mu (u_{n_k}-u_{\o})-(u_{n_k}^{q}-u_{\o}^{q}) \quad \hbox{in}\;\; \O.
\end{equation*}
Equivalently,
\begin{equation*}
    	-\text{div}\left(\frac{\nabla(u_{n_{k}}-u_{\o})}{\sqrt{1-\l_{\o}|\nabla u_{\o}|^{2}}}\right) = \mu (u_{n_k}-u_{\o})-(u_{n_k}^{q}-u_{\o}^{q}) + \text{div}\left(A_{k}(x)\nabla u_{n_k}\right),
\end{equation*}
where
\begin{equation*}
    	A_{k}:=\frac{1}{\sqrt{1-\l_{n_k}|\nabla u_{n_k}|^{2}}}-\frac{1}{\sqrt{1-\l_{\o}|\nabla u_{\o}|^{2}}}.
\end{equation*}
By some standard $L^p$-elliptic estimates (see, e.g., Gilbarg and Trudinger  \cite{GT}), there is a positive constant $C>0$ such that
\begin{equation*}
		\|u_{n_{k}}-u_{\o}\|_{W_0^{2,p}(\O)}\leq C \left( \|u_{n_{k}}-u_{\o}\|_{L^{p}(\O)}+ \|u^{q}_{n_{k}}-u^{q}_{\o}\|_{L^{p}(\O)}+\|\text{div}(A_{k}\nabla u_{n_k})\|_{L^{p}(\O)}\right)
\end{equation*}
for all $k\in\N$. On the other hand, since $\{u_{n_k}\}_{k\in\N}$ is bounded in $W^{2,p}(\O)$, it is relatively compact in $\mc{C}^{1,\nu}(\bar\O)$. Therefore, letting $k\to \infty$ in the previous estimates yields $$
	\lim_{k\to \infty} (\l_{n_{k}},u_{n_{k}}) = (\l_{\o},u_{\o})\quad \hbox{in}\;\;
	[\l_-,\l_+] \times W^{2,p}_{0}(\O).
$$
This ends the proof.
\end{proof}

In the special case when $\l=0$, the problem \eqref{vi.1} becomes
\begin{equation}
\label{vi.7}
\left\{		\begin{array}{ll}
			-\Delta u 	= \mu u - u^{q} & \text{in } \Omega, \\[1ex]
			u=0 & \text{on } \partial\Omega.
\end{array}		\right.
\end{equation}
For this problem, the next result is already classical (see, e.g., the monograph \cite{LG15} and the references therein).

\begin{lemma}
\label{le6.3}
For every $\mu>\s_1$, the problem \eqref{vi.7} admits a unique positive solution $u_{0}\in W^{2,p}_{0}(\O)$. Moreover,  $D_{u}\mf{F}(0,u_{0})\in GL(W^{2,p}_{0}(\O), L^{p}(\O))$.
\end{lemma}

The next strong positivity result will be very useful later.

\begin{lemma}
\label{le6.4} Any positive solution $u \in W_0^{2,p}(\O)$ of \eqref{vi.1} satisfies $u\gg 0$ in the sense that $u(x)>0$ for all $x\in\O$ and $\frac{\p u}{\p n}(x)<0$ for all $x\in\p\O$, where $n$ stands for the outward
	unit normal to $\O$ along $\p\O$.
\end{lemma}
\begin{proof}
Thanks to a result of Bony \cite{Bo}, the Hopf maximum principle, \cite{Hoa}, and the boundary lemma of Hopf \cite{Hob} and Oleinik \cite{Ol} work out in  $W^{2,p}(\O)$ (see, e.g., \cite{LG13}). 	Suppose $u_{0}
\in W_0^{2,p}(\O)$ is a positive solution of \eqref{vi.1}. Then, $u_{0}$ satisfies the linear elliptic equation
$$
	-\operatorname{div}\!\left(A(x)\nabla u\right)+b(x)u=0\quad \text{in} \;\; \Omega,
$$
where
$$
A(x):=\dfrac{1}{\sqrt{1-\lambda |\nabla u_{0}(x)|^{2}}}, \qquad b(x):=u_{0}^{q-1}(x)-\mu, \quad x\in {\O}.
$$
Thanks to the embedding $W^{2,p}(\O)\hookrightarrow \mathcal{C}^{1,1-\frac{N}{p}}(\bar{\Omega})$, the
underlying differential  operator
$$
    \mc{L}u:=-\operatorname{div}\!\left(A(x)\nabla u\right)+b(x)u
$$
is uniformly elliptic with bounded coefficients. Consider a sufficiently large constant, $\o>0$, such that $c:=b+\o\geq 0$ in $\O$. Then,
$$
	- \operatorname{div}\!\left(A(x)\nabla u_{0}\right)+ c(x)u_{0}=\o u_{0} \gneq 0\quad\hbox{in}\;\;\O.
$$
Thus, since $\min_{\bar\O}u_{0} =0$, it follows from the Hopf minimum principle that $u_0$ cannot reach
the value $0$ in $\O$ unless $u_0=0$. Thus, since $u_0\gneq 0$, necessarily $u_0(x)>0$ for all $x\in\O$. Moreover, since $\O$ is smooth, by the boundary lemma of Hopf \cite{Hob} and Oleinik \cite{Ol}, $\frac{\p u_0}{\p n}(x)<0$ for all $x\in\p\O$. This ends the proof.
\end{proof}

As a byproduct of Lemma \ref{le6.4}, the next result holds.

\begin{lemma}
\label{le6.5}
Let $\{(\l_{n},u_{n})\}_{n\in\N}\subset \mf{F}^{-1}(0)$ be such that $u_{n}\geq 0$ for all  $n\in\N$ and
\begin{equation}
\label{vi.8}
	\lim_{n\to+\infty}(\l_{n},u_n)=(\l_{0},u_0)\in\mf{F}^{-1}(0) \quad \text{in} \;\; \R\times W^{2,p}_{0}(\O).
\end{equation}
Then, either $u_{0}\gg 0$, or $u_{0}=0$.
\end{lemma}
\begin{proof}
By \eqref{vi.8}, it follows from the Sobolev embedding $W^{2,p}(\O)\hookrightarrow \mc{C}^{1,1-\frac{N}{p}}(\bar\O)$ that  $u_{0}$ is the pointwise limit of a sequence of functions $u_{n}\gg 0$, $n\in\mathbb{N}$. Thus, either $u_{0}=0$, or $u_{0}\gneq 0$. In the later case, by Lemma \ref{le6.4},
$u_0\gg 0$. The proof is complete.
\end{proof}

The next result provides us with some a priori bounds for the positive solutions of \eqref{vi.1}.

\begin{lemma}
\label{l.6}
For every positive solution $u\in W^{2,p}_{0}(\O)$ of \eqref{vi.1}, the next estimate holds  	
\begin{equation}
\label{vi.9}
	\|u\|_{\infty}\leq \mu^{\frac{1}{q-1}}.
\end{equation}
\end{lemma}

\begin{proof}
Let $u\in W^{2,p}_{0}(\O)$ be a positive solution of \eqref{vi.1}, fix a $\kappa >\mu^{\frac{1}{q-1}}$, and consider the function
$$
	\phi:=(u-\kappa)_{+}\in W^{1,1}_{0}(\O).
$$
Then, $\nabla \phi = \nabla u$ in $\{u>k\}$, and, considering $\phi$ as a test function, we obtain that
\begin{align*}
	0 \leq \int_{\{u>k\}}\frac{|\nabla u|^{2}}{\sqrt{1-\l |\nabla u|^{2}}} \; dx & = \int_{\{u>k\}}(\mu u- u^{q}) (u-k) \; dx \\
	& \leq \int_{\{u>k\}}(\mu-k^{q-1})u(u-k) \; dx \leq 0.
\end{align*}
Therefore, for every $k>\mu^{\frac{1}{q-1}}$, we have that $\nabla u=0$ a.e. in $\{u>k\}$. Hence, $\nabla(u-k)_{+}=0$ almost everywhere and, since $\phi\in W^{1,1}_{0}(\O)$, we can infer that $\phi\equiv 0$. This implies \eqref{vi.9} and ends the proof.
\end{proof}

Note that, for every $\l>0$, the two estimates
$$
   \|u\|_{\infty}\leq \mu^{\frac{1}{q-1}} \;\; \text{  and  } \;\;  \|\nabla u \|_{\infty} < \sqrt{\tfrac{1-\d}{\l}}
$$
hold. Since the corresponding operator is uniformly elliptic once the gradient estimate is established, by standard $L^p$-regularity theory (see \cite[Ch. 9]{GT}, if necessary), we have the following:

\begin{lemma}
	\label{l.6II}
	Let $\l>0$. Then, every $(\l,u)\in \mc{U}_{\d}$ satisfies the estimate
	\begin{equation}
		\label{vi.9II}
		\|u\|_{W^{2,p}(\O)}\leq C,
	\end{equation}
for some positive constant $C>0$ depending on $\l$.
\end{lemma}

Subsequently, we consider  $\mathscr{S}_\d:=\mc{U}_\d\cap \mf{F}^{-1}(0)$, and the sets
\[
   \mathcal{U}_{\d, 0}^{+}:=\mathcal{U}_{\d}\cap \big[(0,+\infty)\times W^{2,p}_{0}(\O)\big], \qquad
   \mathcal{U}_{\d,0}^{-}:=\mathcal{U}_{\d}\cap \big[(-\infty,0)\times W^{2,p}_{0}(\O)\big],
\]
\[
\mathscr{S}^{+}_{\d, 0}:=\{(\lambda,u)\in\mathscr{S}_\d:\lambda> 0\}, \qquad
\mathscr{S}^{-}_{\d, 0}:=\{(\lambda,u)\in\mathscr{S}_\d:\lambda< 0\}.
\]
Note that
$$
   \mc{U}_{\d,0}^{-} = (-\infty,0)\times W^{2,p}_{0}(\O)
$$
as a consequence of \eqref{Eq6.2}. Then, as a direct consequence from Theorem \ref{th5.1}, the next result holds.

\begin{theorem}
\label{th6.1}
Suppose $\mu>\s_1$ and $q\geq 2$ is integer. Then, for every $\d\in (0,1)$, there exist $\o^\pm \in \N\cup \{+\infty\}$ and two locally injective continuous curves of positive solutions of \eqref{vi.1},
$$
   \G^{\pm}: (0,\o^\pm)\longrightarrow \mc{U}_{\d, 0}^{\pm}, \qquad \Gamma^{\pm}\left((0,\o^\pm)\right)\subset \mathscr{S}_{\d, 0}^{\pm}, \qquad \lim_{t\da 0}\Gamma^{\pm}(t)=(0,u_{0}),
$$
where $u_0$ is the unique solution of \eqref{vi.7}, such that $\G^+$ satisfies some of the  following  alternatives:
\begin{enumerate}
\item[{\rm (i)}] It holds that $\mc{P}_{\l}(\Gamma^{+}((0,\o^+)))=(0,+\infty)$,  where $\mc{P}_{\l}:\R\times W^{2,p}_{0}(\O)\to \R$ stands for the $\l$-projection operator.
\item[{\rm (ii)}] There exists a sequence $\{t_{n}\}_{n\in\N}$ in $(0,\o^+)$ such that
$$
    \lim_{n\to +\infty}t_n=\o^+\;\;\hbox{and}\;\;
    \lim_{n\to+\infty} \Gamma^{+}(t_{n})= (\l_{\o^+},u_{\o^+})\in \partial\mc{U}_{\d,0}^+.
$$
\end{enumerate}
Similarly, either $\mc{P}_{\l}(\Gamma^{-}((0,\o^-)))=(-\infty,0)$, or
\begin{equation}
	\label{vi.10}
    \limsup_{t\ua \o^-} \|\nabla \mc{P}_u(\Gamma^-(t))\|_{\mc{C}(\bar\O)} = +\infty,
\end{equation}
where $\mc{P}_{u}:\R\times W^{2,p}_{0}(\O)\to W^{2,p}_0(\O)$ is the $u$-projection operator.
\end{theorem}

\begin{proof} Since $q\geq 2$ is an integer, the operator $\mf{F}$ is analytic and the existence of $\G^\pm$ is guaranteed by Theorem \ref{th5.1}. Moreover, each of
these curves satisfies some of the alternatives, (a), (b), or (c), of Theorem \ref{th5.1}. Actually,
since $u_0$ is the unique solution of \eqref{vi.7}, the alternative (c) cannot occur.
\par
Suppose that $\G^+$ satisfies Theorem \ref{th5.1}(a) and that $\mc{P}_{\l}(\Gamma^{+}((0,\o^+)))$ is a proper  subinterval of $(0,+\infty)$. Then, it is necessarily bounded. But this contradicts the estimate \eqref{vi.9II}. If $\G^+$ satisfies Theorem \ref{th5.1}(b), then (ii) holds.
\par
As, for every $\l<0$, $1-\lambda \|\nabla u\|^{2}_{\infty}>1>\d$, it becomes apparent that  $\G^-$ cannot satisfy (b). Thus, it satisfies Part (a) of Theorem \ref{th5.1}.
Suppose that $\mc{P}_{\l}(\Gamma^{-}((0,\o^-)))$ is a proper subinterval of $(-\infty,0)$. Then, it is necessarily bounded and thanks to \eqref{vi.9}, the relation \eqref{vi.10} holds. This concludes the proof.
\end{proof}

Similarly, according to Theorem \ref{th3.2} and playing around with the same ingredients as in the proof
of Theorem \ref{th6.1},  the next result holds in the general case when $q>1$ is not an integer.

\begin{theorem}
\label{th6.2}
Suppose $\mu>\s_1$ and $q\in(1,+\infty)\setminus\N$. Then, for every $\d\in (0,1)$, there are two connected
components, $\mathscr{C}_\d^\pm$, of positive solutions of $\mathscr{S}_\d$, with $(\l,u)=(0,u_0)\in\mathscr{C}_\d^\pm$, such that
$\mathscr{C}_\d^+$ satisfies some of the following alternatives:
\begin{enumerate}
\item[{\rm (i)}] $\mc{P}_{\l}(\mathscr{C}_\d^+)=(0,+\infty)$,
\item[{\rm (ii)}] $\mathscr{C}_\d^+\cap \p \mc{U}_{\d,0}^+\neq \emptyset$.
\end{enumerate}
Similarly, either $\mc{P}_{\l}(\mathscr{C}_\d^-)=(-\infty,0)$, or there exists $\l_*\leq 0$ and a sequence $\{(\l_n,u_n)\}_{n\geq 1}$ in $\mathscr{C}_\d^-$ such that
$$
   \lim_{n\to\infty}\l_n=\l_*\;\;\hbox{and}\;\; \limsup_{n\to \infty} \|\nabla u_n\|_{\mc{C}(\bar\O)} = +\infty.
$$
\end{theorem}

The fact that the gradients of the positive solutions of quasilinear elliptic equations involving the mean curvature operator can develop singularities even when the solutions are bounded is well documented in the literature. See for instance, the classical references of Jenkins--Serrin and Serrin \cite{MC6, Se} (see also \cite[Ch. 16]{GT}), where the mean curvature
of $\partial\O$ must satisfy certain geometric conditions in order to obtain gradient bounds. A one-dimensional example was given in Theorem 3.2 of Cano-Casanova et al. \cite{CLT}. Some one-dimensional examples under Neumann boundary conditions were given by L\'{o}pez-G\'{o}mez and Omari in
\cite{LO1,LO2,LO3}.  In some one-dimensional prototypes, the continua of classical positive solution can be extended by continua of bounded variation solutions (see \cite{LO0}).

\appendix

\section{Elements of analytic varieties}\label{AB}

\noindent In this section we collect some basic concepts and results of the theory of analytic manifolds. The reader is sent to  Buffoni and Toland \cite[Ch. 7]{BT} for any further details and the proofs of the results collected in this appendix.
\par
Throughout this section, $\K\in\{\R,\C\}$.  Given an integer $n\geq 1$, a nonempty open subset $\O$ of $\K^{n}$, and a finite collection, $\mc{F}$, of analytic functions $f:\O \to\K$, the $\K$-analytic variety generated by $\mc{F}$ on $\O$ is the set
$$
    \mathscr{V}(\O,\mc{F}):=\{z\in\O : f(z)=0 \text{ for all } f\in\mc{F}\}.
$$
An analytic map $f:\O\to \K$ is said to be real-on-real if $f(z)\in\R$ for all $z\in \O\cap\R^{n}$. When $\mc{F}$ consists of real-on-real functions and $\O\cap\R^{n}\neq \emptyset$, the $\K$-analytic variety $\mathscr{V}(\O,\mc{F})$ is said to be real-on-real. A point $z\in \mathscr{V}(\O,\mc{F})$ is $m$-regular if there is a neighborhood $\mathcal{U}$ of $z$ in $\K^{n}$ such that $\mc{U}\cap \mathscr{V}(\O,\mc{F})$ is a $\K$-analytic manifold of dimension $m$.
\par
Two subsets, $A$ and $B$, of $\K^{n}$, are said to be equivalent at $z_{0}\in\K^{n}$ if there is an open neighbourhood $\mc{U}$ of $z_{0}$ such that $\mc{U}\cap A=\mc{U}\cap B$. This establishes an equivalence relation on the power set $\mathscr{P}(\K^{n})$. The equivalence class of $A\subset \K^{n}$, denoted by
$\gamma_{z_{0}}(A)$ is called the germ of $A$ at $z_{0}$. The germ at $z_{0}$ of a $\K$-analytic variety is referred to as a $\K$-analytic germ. A germ of a real-on-real $\C$-analytic variety is called a real-on-real germ. The set of all $\K$-analytic germs at $z_{0}\in\K^{n}$ is denoted by $\mathscr{V}_{z_{0}}(\K^{n})$. For any given $\alpha\in\mathscr{V}_{z_{0}}(\K^{n})$, the dimension of $\a$, $\dim_{\K}\alpha$, is the largest integer $m\in\N$ for which every representative of the class $\alpha$ contains an $m$-regular point. If no such integer exists, we set $\dim_{\K}\alpha=-1$.
\par
The main goal of this appendix is to show how analytic varieties can be viewed as zero sets of a Weierstrass polynomials. The next definition introduces this concept. For every $\d>0$ and $m\in \N$, we are denoting
$$
   B_0(\d,m):= \{(z_1,...,z_m)\in\C^m : \; |z_j|<\d,\;1\leq j\leq m\}.
$$

\begin{definition}[\textbf{Weierstrass polynomial}]
\label{deA.1}
A Weierstrass polynomial on $B_0(\d,m)$ is a polynomial $\mc{P}(\l; z_{1},\cdots,z_{m})$ of the form
\begin{equation}
\label{A.1}
		\mc{P}(\l; z_{1},\cdots,z_{m}):=\l^{d}+\sum_{j=0}^{d-1}c_{j}(z_{1},\cdots,z_{m})\l^{j}, \quad (z_{1},\cdots,z_{m})\in B_0(\d,m),
\end{equation}
for some $d\in \N$ and some analytic functions $c_{j}: B_0(\d,m)\to \K$ such that $c_{j}(0)=0$, $0\leq j\leq d-1$, and
$$
   \Delta(z_{1},\cdots,z_{m})\not\equiv 0,
$$
where $\Delta: B_0(\d,m)\to \C$ is the discriminant of \eqref{A.1}.
\end{definition}

The $\K$-varieties associated to Weierstrass polynomials are the Weierstrass varieties introduced in the next definition.

\begin{definition}[\textbf{Weierstrass variety}]
\label{deA.2}
Given $B_{0}(\delta,m)\subset \C^{m}$, $m < n$, and
$$
   \mathscr{W}=\{\mc{P}_{m+1}, \mc{P}_{m+2},\cdots, \mc{P}_{n}\}
$$
a finite set of Weierstrass polynomials on $B_0(\d,m)$ of the form
$$
    \mc{P}_{j}=\mc{P}_{j}(z_{j};z_{1},\cdots,z_{m}), \quad (z_{1},\cdots,z_{m})\in B_0(\d,m),
    \quad m+1\leq j \leq n,
$$
consider the associated family of analytic functions
$$
    \mc{W}=\{h_{m+1},h_{m+2},\cdots,h_{n}\},
$$
where $h_{j}: B_0(\d,m)\times \C^{n-m}\to \C$ are given by
$$
    h_{j}(z_{1},\cdots,z_{n}):=\mc{P}_{j}(z_{j};z_{1},\cdots,z_{m}), \quad m+1\leq j \leq n.
$$
A Weierstrass variety is any subset in $\C^{n}$ of the form $\mathscr{V}(B_0(\d,m)\times \C^{n-m},\mc{W})$.
\end{definition}

Given a Weierstrass variety $\mathscr{V}(B_0(\d,m)\times \C^{n-m},\mc{W})$, the \textit{joint discriminant} of $\mc{W}$, denoted by $\mathscr{D}[\mc{W}]:B_0(\d,m)\to\C$, is defined through
$$
   \mathscr{D}[\mc{W}](z_{1},\cdots,z_{m})=\prod_{j=m+1}^{n}\Delta_{j}(z_{1},\cdots,z_{m}), \quad (z_{1},\cdots,z_{m})\in B_0(\d,m),
$$
where, for every $j\in\{m+1,\cdots,n\}$, $\Delta_{j}:B_0(\d,m)\to \C$ is the discriminant of the Weierstrass polynomial $\mc{P}_{j}$. The \textit{branches} of the variety $\mathscr{V}(B_0(\d,m)\times \C^{n-m},\mc{W})$ are the connected components of
$$
   \mathscr{V}(B_0(\d,m)\times \C^{n-m},\mc{W})\setminus
   \left[\mathscr{V}(B_0(\d,m),\mathscr{D}[\mc{W}])\times \C^{n-m}\right].
$$
The main result concerning the structure of analytic varieties can be stated as follows.

\begin{theorem}[\textbf{Structure theorem}]
\label{thA.1}
Let $n\geq 2$ and $\alpha\in\mathscr{V}_0(\C^{n})\backslash\{0\}$ be such that $\{0\}\subset \alpha\neq \gamma_{0}(\C^{n})$. Then, there exist sets $B_{1},\cdots,B_{N}$ such that:
\begin{enumerate}
\item[{\rm (1)}] $\alpha=\gamma_{0}(B_{1}\cup\cdots\cup B_{N}\cup\{0\})$.
\item[{\rm (2)}]  Each $B_{j}$, $1\leq j\leq N$,  is, after a linear change of coordinates, a branch of a Weierstrass analytic variety.
\item[{\rm (3)}] $\dim_{\C}\alpha=\max\{\dim_{\C} B_{j} : 1\leq j \leq N\}$.
\item[{\rm (4)}] Assume that $L\subset \C^{n}$, with $\gamma_{0}(L)\neq \emptyset$, is a connected $\C$-analytic manifold of dimension $1\leq \ell \leq n$ whose points are $\ell$-regular points of a representative of $\alpha$. Then, there exists $1\leq j\leq N$ such that
$$
    \gamma_{0}(L)\subset \gamma_{0}(\bar{B}_{j}), \quad \dim_{\C}B_{j}=\ell.
$$
\item[{\rm (5)}] If $\alpha$ is real-on-real, then it can be arranged that each branch $B_{j}$ with $B_{j}\cap\R^{n}\neq \emptyset$ is real-on-real.
\item[{\rm (6)}] If $B'_{j}$, $1\leq j\leq K$, denotes the branches which intersects $\R^{n}$ non-trivially, then
		$$\alpha\cap \gamma_{0}(\R^{n})=\gamma_{0}(B'_{1}\cup \cdots \cup B'_{K}\cup\{0\}).$$
\item[{\rm (7)}] $\dim_{\R}(\alpha\cap \R^{n})=\max\{\dim_{\R}(B'_{j}\cap \R^{n}): 1\leq j \leq K\}$.
\end{enumerate}
\end{theorem}

A fundamental tool to parameterize the branches of Weierstrass analytic varieties is the concept of \textit{Puiseux series}. The next result is useful.

\begin{theorem}
\label{thA.2}
	Let $m=1$ and $n\geq 2$. Suppose that $B$ is a branch of the Weierstrass analytic variety
$$
  \mathscr{E}=\mathscr{V}(W\times \C^{n-1},\mc{W}),
$$
where $W$ is chosen so that $\mathscr{D}[\mc{W}]$ is non-zero on $W\setminus \{0\}$. Then, there exist $\ell\in\N$, $\delta>0$ and a $\C$-analytic function
$$
  \psi:\{z\in\C : |z|^{\ell}<\delta\}\longrightarrow \C^{n-1}
$$
such that $\psi(0)=0$ and the mapping
$$
    \Gamma:\{z\in\C : |z|^{\ell}<\delta\}\longrightarrow B\cup\{0\}, \quad z\mapsto (z^{\ell},\psi(z)),
$$
is bijective. 	Moreover, if we assume that $\gamma_{0}(B\cap \R^{n})\notin \{\emptyset, \{0\}\}$, then there exists an integer $\kappa$ such that $0\leq \kappa \leq 2\ell-1$ for which the map
\begin{equation}
\label{A.2}
		\Sigma:(-\delta^{1/\ell},\delta^{1/\ell})\longrightarrow \R^{n}\cap \bar{B}, \quad r\mapsto ((-1)^{\kappa}r^{\ell},\psi(r\exp(\kappa\pi i/\ell))),
\end{equation}
	is bijective.
\end{theorem}

\end{document}